\documentclass[10pt,reqno,oneside]{article}

\usepackage{amsfonts,amssymb,latexsym,color,amsmath,enumerate,stmaryrd,amsthm,ushort,dsfont,mathrsfs,authblk}

\usepackage{hyperref} 

\usepackage[round]{natbib}
\usepackage{graphicx} 
\usepackage{xcolor}
\definecolor{unbleu}{rgb}{0.03, 0.15, 0.4}

\hypersetup{
pdfborder = {0 0 0},
colorlinks,
linkcolor=unbleu,
citecolor=unbleu,
urlcolor=unbleu
}

\usepackage{scalerel} 

\usepackage{geometry}
\usepackage{marginnote}

\usepackage{savesym}
\savesymbol{checkmark}
\usepackage{dingbat}

\numberwithin{equation}{section}

\newcommand{\E}{\mathbb{E}}

\newcommand{\N}{\ensuremath{\mathbb{N}}}    
\newcommand{\R}{\ensuremath{\mathbb{R}}}

\newcommand{\Q}{\ensuremath{\mathbb{Q}}} 

\newcommand{\1}{\ensuremath{\mathds{1}}}

\newcommand{\apud}[2]{(\@apud#1,#2\@endapud)}
\def\@apud#1,#2\@endapud{%
   \cite{#1} \textbf{apud} \cite{#2}}%
\newtheorem{rem}{Remark}[section]
\newtheorem{lemma}{Lemma}[section]

\newtheorem{theo}{Theorem}[section]

\newtheorem{prop}{Proposition}[section]

\begin{document}


\title{Consistent model selection for the Degree Corrected Stochastic Blockmodel}
\author{Andressa Cerqueira, Sandro Gallo, Florencia Leonardi and Cristel Vera}

\maketitle

\begin{abstract}  

The Degree Corrected Stochastic Block Model (DCSBM) was introduced by \cite{karrer2011stochastic} as a generalization of the stochastic block model in which vertices of the same community are allowed to have distinct degree distributions. On the modelling side, this variability makes the DCSBM more suitable for real life complex networks. On the statistical side, it is more challenging due to the large number of parameters when dealing with community detection. 
In this paper we prove that the penalized marginal likelihood estimator is strongly consistent for the estimation of the number of communities. We consider \emph{dense} or \emph{semi-sparse} random networks, and our estimator is \emph{unbounded}, in the sense that the number of communities $k$ considered can be as big as $n$, the number of nodes in the network. 
\end{abstract}


\section{Introduction}

Many real-world phenomena can be described by the interaction of objects through a network. For example, interactions between individuals in a social network, connections between airports in a country, connections between regions of the brain, etc. Most of these networks have a community structure; that is, the objects (nodes of the network) belonging to the same group tend to behave similarly. In this way, probabilistic models that aim to describe real networks need to incorporate these community structures.

The Stochastic Block Model (SBM) proposed by \cite{holland1983stochastic} is a random network model allowing community structures. Each pair of vertices is connected, independently of everything, by a Poisson number of edges whose rate depends solely on the communities of the involved vertices. The SBM, therefore, models networks where nodes in the same community have the same mean.  This property can restrict applications to real-life networks that, not rarely, display heterogeneity (\emph{hubs}) in the degree distributions of vertices belonging to the same community. Taking this into account,  \cite{karrer2011stochastic} proposed the Degree-Corrected Stochastic Block Model (DCSBM), which considers the heterogeneity in the nodes' degrees within communities. In the degree-corrected model, each node has associated a non-negative real parameter, a weight, specifying its ``ability'' to connect to other nodes in the network. The sum of the weights in each community corresponds to the number of nodes belonging to the community, generalizing the homogeneous SBM where we can consider each node as having a weight equal to one. In both, the SBM and the DCSBM, it is a standard approach to study two different regimes, the \emph{dense} and the \emph{semi-sparse} regimes. In the former, the rates of the Poisson distribution governing the number of edges between each pair of nodes are fixed (do not depend on the number $n$ of nodes in the networks), leading to linearly growing expected degrees for each node. In the latter, these rates are allowed to decrease to zero in such a way that the expected degrees of the nodes grow much slower than $n$.

Several works in the literature have addressed the community detection problem for  SBM and  DCSBM, where the goal is to estimate the $k_0$ latent groups of nodes in the network. For the SBM, community detection is proposed based on spectral methods \citep{Rohe2011spectral,Lei2015detection,Sarkar2015}, modularity \citep{Girvan2004Finding},   likelihood methods \citep{bickel2009nonparametric,celisse2012consistency,amini2013pseudo} and under a Bayesian perspective \citep{Decelle2011,latouche2012variational,van2017bayesian}.
For the DCSBM, \cite{Zhao2012} study consistency of modularity-based and likelihood-based methods, \cite{qin2013} proposed a regularized spectral clustering algorithm and \cite{Jin2015} proposed an approach based on the entry-wise ratios between eigenvectors of the adjacency matrix. In order to select the best model between the SBM and DCSBM to fit the data,  \cite{yan2014model} proposed an approach based on the likelihood ratio test computed approximately using belief propagation. All these methods assume the number of communities $k_0$ is known, something that barely occurs in practice. Estimating the number of communities can be considered a model selection problem. 

The literature on estimating the number of $k_0$ communities is more recent and not that extensive, at least from the theoretical point of view. In the case of standard SBM, some approaches include sequential hypothesis tests \citep{lei2016goodness}, cross-validation \citep{chen2018}, spectral methods \citep{Le2022}, penalized likelihood criteria \citep{wang2017likelihood,hu2020corrected} and penalized marginal likelihood estimators \citep{daudin2008mixture,biernacki2010exact,latouche2012variational,cerqueira-leonardi-2020}. 
Specifically for the DCSBM with $n$ nodes and unknown weights and under the semi-sparse regime with $n^{1/2}\rho_n/\log n\to \infty$, \cite{wang2017likelihood} proved the consistency of the penalized likelihood estimator with a penalty function of order $k^2 n\log n$ where $k$ is the number of communities of the candidate model.
More recently, \cite{ma2021determining} proposed a likelihood ratio test to estimate the number of communities and proved its consistency for the semi-sparse regime where $n\rho_n/\log n$ is sufficiently large. Their approach is based on spectral algorithms and so they assume many {further} hypotheses to correctly detect the groups. 
%
Both approaches assume the number of communities is bounded from above by a known constant. 

The present paper considers the penalized marginal likelihood estimator for the number of communities under a DCSBM with unknown weights, proposed by \citet{cerqueira-leonardi-2020} for the standard SBM. This estimator can be seen as a \emph{minimum-description length} principle, and it is also known as \emph{Krichevsky-Trofimov} estimator in the information theory community. We prove that this estimator equals the correct number of communities $k_0$ asymptotically almost-surely (for a sufficiently large number of vertices $n$ with probability one) under the more general degree-corrected model and without assuming an upper bound for $k_0$. That is, the optimization is made over all possible numbers of communities between 1 and $n$. As pointed out above, the degree-corrected model has $n$ extra parameters, due to weights associated to each node. For this reason, we need an additional term in the penalty function of order $kn\log n$ with respect to the penalty function for the standard SBM used in \citet{cerqueira-leonardi-2020}. We consider the same semi-sparse regime of \cite{ma2021determining}, where $n\rho_n/\log n$ is sufficiently large, this rate being the phase transition for exact recovery of the communities (see \citet{abbe2018}). 

The paper is organized as follows. We define the DCSBM and its associated likelihood function (for known parameters) in Section \ref{sec:DCSBMdef}. In Section \ref{sec:main_results},  we introduce \emph{a priori} distributions for the parameters, define the penalized marginal likelihood estimator, and state our main theorem, the consistency results, as well as a key proposition relating the marginal likelihood and the maximum likelihood. Finally, in Section \ref{sec:provas} we present the proof of the main result.
Technical proofs and other auxiliary results are deferred to the appendix.

\section{The Degree Corrected Stochastic Block Model: definition and likelihood}\label{sec:DCSBMdef}

For any $n\in\N$, let  $X=(X_{ij})_{i,j\in[n]}$ ($[n]:=\{1,\ldots,n\}$) denote the symmetric adjacency matrix of a random network on $n$ vertices, with $X_{ij}\in\N$. 
For each pair  $i,j\in [n]$, with $i\neq j$ the variable $X_{ij}$ 
represents the number of non-oriented edges (or alternatively,  the strength of connection) between vertices $i$ and $j$. 
For convenience, we define  $X_{ii}$ as two times the number of self-loops at vertex $i$.

The vertices are randomly divided into $k_0\ge1$ communities and this community attribution is represented by the vector $Z= (Z_1,\ldots,Z_n)$ of $[k_0]$-valued random variables (i.i.d. with marginal distribution $\pi$). We will often use the notation $i\in[a]$ ($i\in[n],a\in[k_0]$) to mean that $Z_i=a$. 

In the homogeneous SBM,  the expected number of edges between vertices $i$ and $j$ does not depend on the specific vertices but only on the communities. Assuming the number of edges between communities $a,b$ has a Poisson distribution with parameter $\lambda_{ab}$  we have that 
\[
\E(X_{ij} | Z_i=z_i, Z_j=z_j) =  \lambda_{z_iz_j} 
\] 
for any $i\neq j$.

The fact that, within each community, the vertices behave identically, is a disadvantage of the SBM when modeling real-world complex networks. In order to allow different vertices to behave differently inside each community, the degree-corrected SBM (DCSBM) incorporates a weight $w_i$ for each vertex $i\in [n]$ which influences the capacity of the vertex to connect to other vertices.
In this case, the expected number of edges between vertices $i\neq j$  is given by 
\[
\E(X_{ij} | Z_i=z_i, Z_j=z_j) =  w_iw_j  \lambda_{z_iz_j}
\] 
and may be different for different nodes in the same community. For this reason such networks are sometimes called \emph{inhomogeneous}. 

We assume the matrix $\lambda=(\lambda_{ab})_{a,b\in[k_0]}$ is symmetric and has all  entries greater than zero.
In the \emph{dense} regime, the matrix $\lambda$ is fixed (does not depend on $n$) and has all its entries bounded from below by a positive constant. In this case each node has an expected degree that grows linearly on $n$, which makes the network over-connected. For this reason, it is interesting to consider a \emph{semi-sparse} regime, where $\lambda$ is allowed to decrease to zero as a function of $n$. We take this approach here and we assume that for each $n$, the distribution of the network on $n$ nodes has parameter $\lambda = \rho_n \tilde\lambda$, with $\rho_n\to0$ as $n\to\infty$ and $\tilde\lambda$  a constant symmetric matrix with entries bounded from below by a positive constant. We will give conditions on $\rho_n$ on our main results. 
For identifiability of the order $k_0$, i.e the number of communities of the model,  we assume that no column in $\tilde\lambda$ is proportional to any other column. This is usually assumed in the literature, see for example \citet{ma2021determining}. 

Consider now the DCSBM with $k\in[n]$ communities. In order to compute the joint distribution of $(Z,X)$ we need to define the following counters. For any $z\in [k]^n$ and $a\in[k]$, let  $n_a(z)$ be the number of vertices in the network that belong to community $a$, that is
\[
n_a(z) = \sum_{i=1}^n \mathds{1}\{z_i=a\}\,.
\]
Following \cite{karrer2011stochastic} we assume that the vector of weights inside each community satisfies 
\[
\sum_{i\colon z_i=a} w_i \;=\;n_a(z)
\]
for all $a\in [k]$. This implies in particular that the total weight in the network is $n$, the number of vertices, and putting $w_i\equiv1$ we retrieve the classical SBM.

For any symmetric matrix $x\in \N^{n\times n}$ (that is, any realization of the network) define also the counter $o_{ab}(x,z)$ as the number of edges between nodes of communities $a$ and $b$, that is 
\begin{equation}
o_{ab}(x,z)= \begin{cases}
\sum_{1\leq i, j\leq n}x_{ij}\mathds{1}\lbrace z_i=a,z_j=b\rbrace & \text{ for }a\neq b\,;\\
\frac12\sum_{1\leq i,j\leq n}x_{ij}\mathds{1}\lbrace z_i=a,z_j=a\rbrace & \text{ for }a=b
\end{cases}
\end{equation}
and the degree of node $i$  by 
\begin{equation}
d_i(x)=\sum_{1\leq j\leq n}x_{ij}.
\end{equation}
The total degree on community $a$ is denoted by $d_a^t(x,z)$, and is given by
\[
d_a^t(x,z) = \sum_{i\colon z_i=a} d_i(x) = \sum_{1\leq i\leq n} d_i(x)\1\{z_i=a\}=
 \sum_{1\leq i,j\leq n} x_{ij}\1\{z_i=a\} \,.
\]
Observe that we have  
\[
d_a^t(x,z)  =  \sum_{b\neq a} o_{ab}(x,z) + 2o_{aa}(x,z)\,
\]
and that the number of pairs of nodes in communities $a$ and $b$, denoted by $n_{ab}(z)$, is given by 
\begin{equation}
n_{ab}(z)= \begin{cases}
n_a(z)n_b(z) & \text{ for }a\neq b\,;\\
\frac12 n_a(z)^2 & \text{ for }a=b.
\end{cases}
\end{equation}

We can now write down the joint distribution of $(Z,X)$ for the model with $k$ communities when we are given all the other parameters
\begin{equation}\label{eq:conjunta}
p(z,x|\pi,\lambda,w)=p(z|\pi)p(x|z,w,\lambda)
\end{equation}
where 
\begin{align}\label{eq:z}
p(z|\pi)=\prod_{1\leq a\leq k} \pi_a^{n_a(z)}  
\end{align}
and
\begin{align}\label{eq:x|w,z}
p(x|z,w,\lambda)=\dfrac{1}{c(x)}\Biggl[\prod_{1\leq i\leq n}{w_i}^{d_i(x)}\Biggr]\Biggl[\prod_{1\leq a\leq b\leq k}\lambda_{ab}^{o_{ab}(x,z)}\exp\lbrace - 
n_{ab}(z) \lambda_{ab}\rbrace\Biggr]
\end{align}
where 
\begin{equation}\label{eq:c}
c(x):=\Biggl[\prod_{i<j}x_{ij}!\Biggl]\Biggl[\prod_{i}2^{x_{ii}/2}(x_{ii}/2)!\Biggl]\,.
\end{equation}

\section{Model selection: Main result}\label{sec:main_results}

To define the estimator, we introduce convenient \emph{a priori} distributions for the parameters  $(\pi,\lambda, w)$. Then the hierarchical model distribution of the DCSBM is given by 
\begin{align*}
\pi &\;\,\sim\,\; \text{Dirichlet}(\overbrace{{1}/{2},\dots, {1}/{2}}^{k})\\
\lambda_{ab} &\; \sim\; 
\text{Gamma}(1/2,  1)\,,\quad  \text{ for } a,b\in [k],\, a \leq  b\\
z_i \,|\, \pi &\; \sim \;\pi\,,\quad i\in [n]\\
( w_i)_{i\in [a]}\,|\, z  &\; \sim \; n_a(z) \; \text{Dirichlet}(\overbrace{1/2, \dots,  1/2}^{n_a(z)})\,,\quad a\in [k]\\
x_{ij}\,|\,z_i, z_j,  w_i,  w_j,\lambda&\;\sim\; 
(2-\mathds{1}\{i\neq j\}) \;\text{Poisson}( w_i w_j\lambda_{z_i,z_j})\,,  \quad \text{ for }i,j\in [n],\, i\leq j\,.
\end{align*}
Let $(\lambda_{ab})_{a,b\in [k]}$ be the symmetric matrix constructed from $(\lambda_{ab})_{1\leq a\leq b\leq k}$. 
Denote by $\Theta_k$ the space where the hyperparameters $(\pi,\lambda,w)$ take values and by $\nu_k(\theta)$ the  a priori distribution  over $\Theta_k$.

For any $x\in \N^{n\times n}$, the marginal likelihood  $p_k(x)$ is given by the integral 
\begin{align*}
p_k(x)&=
\sum_{z\in[k]^n} \int_{\Theta_k}  p(x,z|\theta)\nu_k(\theta)d\theta 
\end{align*}
where $\nu_k$ is explicitly decomposed as in \eqref{eq:priori} (in the appendix).
We can now define the estimator for the number of communities as
\begin{equation}\label{eqd5}
\widehat{k}_n(x):=\underset{1\leq k \leq n}{\arg\max}\left\lbrace\;\log p_k(x)- (k^3+3kn) \log (n+1)\;\right\rbrace\,.
\end{equation}

\begin{rem}
Observe that the estimator $\hat k_n$ in \eqref{eqd5} has an extra penalty term of order $n\log n$ with respect to the estimator defined in \citet{cerqueira-leonardi-2020}. This extra term is needed by the addition of $n$ parameters $w_i$, $i=1,\dots,n$ to the model, and is unnecessary in the homogeneous SBM.
\end{rem}

We now state our main theorem.

\begin{theo}\label{teod1}
For the  DCSBM with $k_0$ communities and $\rho_n\geq C\frac{\log n}{n},n\ge1$, where $C$ is  a sufficiently large constant, the estimator defined in \eqref{eqd5} satisfies 
\begin{equation}
\widehat{k}_n=k_0
\end{equation}
eventually almost surely as $n$ diverges. 
\end{theo}

The proof of Theorem~\ref{teod1} is based on a key proposition (see below) relating the marginal likelihood $p_k(x)$ with the maximum likelihood $\sup_{\theta\in\Theta_k} p(x|\theta)=\sup_{\pi} \; p(z|\pi)\sup_{\lambda,w} \; p(x|z,w,\lambda) $, that we now define. For any $x\in \N^{n\times n},z\in[k]^n, i\in[n]$ and $a,b\in[k]$, the maximum likelihood estimators of $\pi$, $\lambda$ and $w$ are given by 
\[
\widehat\pi_a(z) = \frac{n_a(z)}{n} \,,\quad 
\widehat\lambda_{ab}(x,z) = \frac{o_{ab}(x,z)}{n_{ab}(z)}\,,\quad
\widehat w_i(x,z) = \sum_{1\leq a\leq k} \1\{z_i=a\}\frac{n_a(z)d_i(x)}{d^t_a(x,z)}\,,
\]
and thus
\begin{equation}
\begin{split}
\sup_{\pi} \; p(z|\pi) &= \prod_{1\leq a\leq k} \left(  \frac{n_a(z)}{n}\right)^{n_a(z)}
\end{split}
\end{equation}
and
\begin{equation}
\begin{split}
\sup_{\lambda,w} \; p(x|z,w,\lambda) 
&= {1\over c(x)}  \prod_{1\leq a\leq b\leq k}\left( \frac{o_{ab}(x,z)}{n_{ab}(z)}\right)^{o_{ab}(x,z)}e^{- o_{ab}(x,z)}\\
&\qquad \times \prod_{a\in[k],i\in[n]:z_i=a}\left(  \frac{n_a(z) d_i(x)}{d_a^t(x,z)}\right)^{d_i(x)}\,.
\end{split}
\end{equation}

We now state the key proposition, which holds for a set of ``good'' networks (which will be proved to hold with high probability) defined by 
\begin{equation}\label{Omega}
\Omega_n:=\lbrace {x} \colon  x_{ij}\leq\log n, \text{ for all }i,j\in [n]\rbrace\,.
\end{equation}

\begin{prop}\label{propd1}
For all  
$k\geq 1$, all $n\geq \max(k,3)$ and all ${x}\in \Omega_n$ we have that 
\begin{equation*}
0\;\leq\; \log\left(\dfrac{\sup_{\theta \in \Theta_k} p(x|\theta)}{p_k(x)}\right)\;\leq \; k(k+2)\log (n+1) + 3n\log n \,.
\end{equation*}
\end{prop}

The proof of this proposition is given in the appendix.

\section{Proof of the main result}\label{sec:provas}

The proof of Theorem~\ref{teod1} is split into two parts, first, we prove the non-underestimation of the number of communities and then the non-overestimation. Additional results and definitions are given in the appendix. 

We begin by presenting the proof that the estimator $\widehat{k}_n$ does not underestimate $k_0$, the true number of communities.

\begin{prop}\label{propd3}
For the sparse DCSBM with $\rho_n \geq \frac{C\log n}n$, for a sufficiently large constant $C>0$,  the estimator $\widehat{k}_n$ defined in \eqref{eqd5} satisfies 
\[
\widehat{k}_n\geq k_0,
\] 
eventually almost surely as $n\to\infty$.
\end{prop}

\begin{proof}
We define the profile estimator for the communities based on the observed graph under the model with $k$
communities as 
\begin{equation}\label{equasub6}
\widehat z_k=\underset{z\in{\lbrace 1,2,\ldots, k \rbrace}^{n}}{\arg\max}\sup_{\theta\in\Theta_k}p(x,z| \theta)\,. %
\end{equation}
By Lemma~\ref{lema:concentracao} we can (and henceforth will) take $n$ sufficiently large so that $x\in \Omega_n$, the set of ``good'' networks defined in \eqref{Omega}. 
In order to show that $\widehat{k}_{n}{(x)}\geq k_0$, almost surely when  $n\to\infty$,  it is sufficient to show that for all  $k<k_0$,
\begin{equation}\label{equasub1}
\log p_{k_0}(x)- (k_0^3+3k_0n)\log(n+1)\;>\;\log p_k(x)- (k^3+3kn)\log(n+1)\,,
\end{equation}
almost surely, when $n\to\infty$.
But, if we show that 
\begin{equation}\label{equce}
\liminf_{n\to\infty}\;{1\over \rho_n n^2}\log{p_{k_0}(x)\over  p_{k}(x)}>0\,,
\end{equation}
and due to the fact that 
\begin{equation*}
{1\over \rho_n n^2}\left( k_0^3+3k_0n -k^3-3kn\right)\log(n+1)
\end{equation*}
can be made sufficiently small by assumption on $\rho_n$, it follows that  \eqref{equce} implies \eqref{equasub1}.
Observe that for all $x\in\Omega_n$, by Proposition~\ref{propd1} and the fact that $p_k(x)\leq \sup_{\theta\in \Theta_k}p(x|\theta)$, we have that 
\begin{align*}
{1\over \rho_n n^2}\log{p_{k_0}(x)\over p_k(x)}&={1\over \rho_n n^2}\log{p_{k_0}(x)\over \sup_{\theta\in \Theta_{k_0}}p(x|\theta)}+{1\over \rho_n n^2}\log {\sup_{\theta\in \Theta_{k_0}}p(x|\theta)\over \sup_{\theta\in \Theta_k}p(x|\theta)}+{1\over \rho_n n^2}\log{\sup_{\theta\in \Theta_k}p(x|\theta)\over p_k(x)}\\[2mm]
&\geq - \gamma(k_0){\log n\over \rho_n n}+{1\over \rho_n n^2}\log {\sup_{\theta\in \Theta_{k_0}}p(x|\theta)\over \sup_{\theta\in \Theta_k}p(x|\theta)}\,,
\end{align*}
for some constant $\gamma(k_0)$. 
Then, to show  \eqref{equce} it is enough to prove that for $k<k_0$
\begin{equation}\label{equasub3}
\liminf_{n\to\infty} \; {1\over \rho_n n^2}\log {\sup_{\theta\in \Theta_{k_0}}p(x|\theta)\over \sup_{\theta\in \Theta_k}p(x|\theta)}>0
\end{equation}
as $- \gamma(k_0){\log n\over \rho_n n}$ is also sufficiently small by hypothesis on $\rho_n$. First observe that for all $z\in [k_0]^n$ we have that  
\begin{equation}\label{equa318}
\begin{split}
\log \sup_{\theta\in \Theta_{k_0}}p(x|\theta) 
\;\geq\;
\log \sup_{\theta\in \Theta_{k_0}}p(x, z|\theta)\,.
\end{split}
\end{equation}
Let
\[
\tilde o_{ab}(x,z): = \sum_{1\leq i,j\leq n} x_{ij}\1\{z_i=a,z_j=b\}
\]
for all pairs $a,b$ and notice is relates to $o_{ab}$ by  $\tilde o_{ab}(x,z) = o_{ab}(x,z)$ for all  $a\neq b$ and $\tilde o_{aa}(x,z) = 2o_{aa}(x,z)$ for all $a$.
From \eqref{eq:conjunta}-\eqref{eq:c} and the definition of the maximum likelihood estimators we have that 
\begin{equation}\label{equasub7}
\begin{split}
\log\sup_{\theta\in\Theta_{k_0}}p(x,z|\theta)&\;=\;  L(x) 
+ n\sum_{a=1}^{k_0}\widehat{\pi}_a(z)\log\widehat{\pi}_a(z) + 
 \frac12 \sum_{1\leq a, b\leq k_0} \tilde o_{ab}(x,z)\log \widehat{\lambda}_{ab}(x,z)\\
 &\qquad +  \sum_{1\leq a\leq k_0} d^t_a(x,z) \log \frac{n_a(z)}{d^t_a(x,z)}
\end{split}
\end{equation}
with 
\begin{align*}
 L(x) &\;=\;  -\log c(x) - \sum_{i,j} x_{ij} + \sum_{i} d_i(x)\log d_i(x)\,; \\ 
\widehat{\pi}_a(z)&\;=\;{n_a(z)\over n}, \qquad\qquad 1\leq a\leq k_0;\\
\widehat{\lambda}_{ab}(x,z)&\;=\;{\tilde o_{ab}(x,z)\over n_a(z)n_b(z)}, \qquad 1\leq a, b\leq k_0;\\
d_a^t(x,z) &\;=\;  \sum_{1\leq i,j\leq n} x_{ij}\1\{z_i=a\}\;=\; \sum_{b} \tilde o_{ab}(x,z),  \qquad 1\leq a\leq k_0\,.
\end{align*}
For the denominator in \eqref{equasub3} we use that 
\begin{equation}\label{equa320}
\begin{split}
 \log \sup_{\theta\in \Theta_k}p(x|\theta) 
\;&\leq\; \log k^n \sup_{\theta\in \Theta_k}p(x, \widehat z_k|\theta)\\
&\leq\;  n \log k + \log  \sup_{\theta\in \Theta_k}p(x, \widehat z_k|\theta)\,,
\end{split}
\end{equation}
with $\widehat z_k$ defined by  \eqref{equasub6}. Analogously as in  \eqref{equasub7} we have that  
\begin{equation}\label{equasub8}
\begin{split}
\log\sup_{\theta\in\Theta_k}p(x,\widehat z_k|\theta)&\;=\; L(x) 
+n \sum_{1\leq a\leq k} \widehat{\pi}_a(\widehat z_k) \log\widehat{\pi}_a(\widehat z_k) + \frac12 \sum_{1\leq a,  b\leq k} \tilde o_{ab}(\widehat z_k)\log \widehat{\lambda}_{ab}(x,\widehat z_k)\\
&\qquad +  \sum_{1\leq a\leq k} d^t_a(x,\widehat z_k) \log \frac{n_a(\widehat z_k)}{d^t_a(x,\widehat z_k)}\,.
\end{split}
\end{equation}
Then, the logarithm in  \eqref{equasub3} can be lower bounded by the difference of \eqref{equa318} and \eqref{equa320},  and using the expressions in  \eqref{equasub7} and \eqref{equasub8} we obtain that 
\begin{equation}
\begin{split}\label{eq:sixterms}
\log {\sup_{\theta\in \Theta_{k_0}}p(x|\theta)\over \sup_{\theta\in \Theta_k}p(x|\theta)} \;&\geq\;n\sum_{1\leq a\leq k_0}\widehat{\pi}_a(z)\log\widehat{\pi}_a(z) + \frac12 \sum_{1\leq a,  b\leq k_0}\tilde o_{ab}(x,z)\log \widehat{\lambda}_{ab}(x,z)\\[2mm]
&\qquad  +  \sum_{1\leq a\leq k_0}  d_a^t(x,z) \log \frac{n_a(z)}{ d_a^t(x,z)} - n \log k \\[2mm]
&\qquad- n\sum_{1\leq a\leq k}\widehat{\pi}_a(\widehat z_k)\log\widehat{\pi}_a(\widehat z_k) -
\frac12\sum_{1\leq a, b\leq k}\tilde o_{ab}(x,\widehat z_k)\log \widehat{\lambda}_{ab}(x,\widehat z_k)\\[2mm]
&\qquad  -  \sum_{1\leq a\leq k} d^t_a(x,\widehat z_k) \log \frac{n_a(\widehat z_k)}{d^t_a(x,\widehat z_k)}\,.
\end{split}
\end{equation}
We will now rearrange the six terms of the right-hand side. First, let 
\begin{equation}\label{eq1rearrange}
\mathcal A(n):=n\sum_{1\leq a\leq k_0}\widehat{\pi}_a(z)\log\widehat{\pi}_a(z) - n\log k -  n\sum_{1\leq a\leq k}\widehat{\pi}_a(\widehat z_k)\log\widehat{\pi}_a(\widehat z_k).
\end{equation}
Second, we can write
\begin{equation}
\begin{split}
\sum_{1\leq a\leq k_0}  d_a^t(x,z) \log \frac{n_a(z)}{ d_a^t(x,z)} \;&=\; 
\frac12\sum_{1\leq a\leq k_0}  d_a^t(x,z) \log \frac{n_a(z)}{ d_a^t(x,z)} +\frac12\sum_{1\leq b\leq k_0} d^t_b(x,z) \log \frac{n_b(z)}{d^t_b(x,z)}\\
&=\;\frac12 \sum_{1\leq a,b\leq k_0}  \tilde o_{ab}(x,z) \log \frac{n_{a}(z)n_b(z)}{ d_a^t(x,z)d^t_b(x,z)} 
\end{split}
\end{equation}
and therefore
\begin{equation}\label{eq2rearrange}
\begin{split}
\frac12\sum_{1\leq a, b\leq k_0} \tilde o_{ab}(x,z)\log \widehat{\lambda}_{ab}(x,z)&+  \sum_{1\leq a\leq k_0}  d_a^t(x,z) \log \frac{n_a(z)}{ d_a^t(x,z)} =
\frac12\sum_{1\leq a,b\leq k_0} \tilde o_{ab}(x,z) \log \frac{\tilde o_{ab}(x,z)}{d_a^t(x,z)d^t_b(x,z)}.
\end{split}
\end{equation}
Using \eqref{eq1rearrange}, \eqref{eq2rearrange} and the counterpart of \eqref{eq2rearrange} under the $k$-th order model with $\widehat z_k$ instead of $z$, the right-hand side of \eqref{eq:sixterms} now reads 
\begin{align*}
 \log {\sup_{\theta\in \Theta_{k_0}}p(x|\theta)\over \sup_{\theta\in \Theta_k}p(x|\theta)}\;\ge\; \mathcal A(n)+\frac12&\sum_{1\leq a,b\leq k_0} \tilde o_{ab}(x,z) \log \frac{\tilde o_{ab}(x,z)}{d_a^t(x,z)d^t_b(x,z)}\\&-\frac12\sum_{1\leq a,b\leq k} \tilde o_{ab}(x,\widehat z_k) \log \frac{\tilde o_{ab}(x, \widehat z_k)}{d_a^t(x,\widehat z_k)d^t_b(x,\widehat z_k)}\,.
\end{align*}
Now, dividing both sides of \eqref{eq:sixterms} by $\rho_nn^2$, and summing on the right-hand side the following term (which equals 0) 
\[
\frac{1}{\rho_nn^2}\left(\frac12\sum_{1\leq a,b\leq k_0} \tilde o_{ab}(x,z) \log \rho_nn^2 - \frac12\sum_{1\leq a,b\leq k} \tilde o_{ab}(x,\widehat z_k) \log \rho_nn^2\right)
\]
we finally obtain that 
\begin{equation}\label{eq121}
\begin{split}
\frac{1}{\rho_n n^2}\log {\sup_{\theta\in \Theta_{k_0}}p(x|\theta)\over \sup_{\theta\in \Theta_k}p(x|\theta)} \;&\geq\;
\frac12\, \Biggl(\sum_{1\leq a,b\leq k_0} \frac{\tilde o_{ab}(x,z)}{\rho_nn^2} \log \frac{\rho_nn^2\tilde o_{ab}(x,z)}{d_a^t(x,z)d^t_b(x,z)}\\
&\qquad\quad - \sum_{1\leq a,b\leq k} \frac{\tilde o_{ab}(x,\widehat z_k)}{\rho_nn^2} \log \frac{\rho_nn^2\tilde o_{ab}(x,\widehat z_k)}{d^t_a(x,\widehat z_k)d^t_b(x,\widehat z_k)}\Biggr) +\frac{\mathcal A(n)}{\rho_nn^2}.
\end{split}
\end{equation}
Since $\frac{\mathcal A(n)}{\rho_nn^2}$ converges almost surely to $0$, proving that \eqref{eq121} is bounded from below by a positive constant, eventually almost surely as $n\to\infty$, is equivalent to proving that 
\begin{equation}\label{equasub891}
\begin{split}
\liminf_{n\to\infty} \; \sum_{1\leq a, b\leq k_0} &\frac{d_a^t(x,z)d^t_b(x,z)}{\rho_n^2n^4}\varphi\Bigl(\frac{\rho_nn^2\tilde o_{ab}(x,z)}{d_a^t(x,z)d^t_b(x,z)}\Bigr) \\
&- \sum_{1\leq a,  b\leq k}
\frac{d^t_a(x,\widehat z_k)d^t_b(x,\widehat z_k)}{\rho_n^2n^4}\varphi\Bigl(\frac{\rho_nn^2\tilde o_{ab}(x,\widehat z_k)}{d^t_a(x,\widehat z_k)d^t_b(x,\widehat z_k)}\Bigr)\;>\; 0\,,
\end{split}
\end{equation}
with $\varphi(u) = u\log u$. 
By Lemma~\ref{lema-sandro} we have that 
\begin{equation}
\begin{split}
\frac{\rho_nn^2 \tilde o_{ab}(x,z)}{d_a^t(x,z)d^t_b(x,z)} \;=\;\frac{\tilde o_{ab}(x,z)/\rho_n n^2}{d_a^t(x,z)d^t_b(x,z)/\rho_n^2n^4} \;  &\to\;  \frac{ [\text{diag}(\pi)\tilde\lambda\text{diag}(\pi)^T]_{ab} }{[\text{diag}(\pi)\tilde\lambda\text{diag}(\pi)^T{\bf 1}_k]_a[\text{diag}(\pi)\tilde\lambda\text{diag}(\pi)^T{\bf 1}_k]_b} \\
&=\; \frac{ \pi_a\pi_b\tilde\lambda_{ab}}{\pi_a [\tilde\lambda\pi]_b \pi_b [\tilde\lambda \pi]_a}\\
&=\; \frac{\tilde\lambda_{ab}}{[\tilde\lambda\pi]_a [\tilde\lambda \pi]_b}\,
\end{split}
\end{equation}
where we recall that $\tilde\lambda$ is the matrix such that $\lambda = \rho_n \tilde\lambda$. Then we have that 
\begin{equation}\label{equasub9}
\begin{split}
\lim_{n\to\infty} \;  \frac12 \,\sum_{1\leq a, b\leq k_0}& \frac{d_a^t(x,z)d^t_b(x,z)}{\rho_n^2n^4}\;\varphi\Bigl(\frac{ \rho_nn^2 \tilde o_{ab}(x,z)}{d_a^t(x,z)d^t_b(x,z)}\Bigr)  \\
&\;=\; 
 \frac12 \sum_{1\leq a,  b\leq k_0}\pi_a\pi_b[\tilde\lambda\pi]_a [\tilde\lambda \pi]_b\;\varphi\Bigl(\frac{\tilde\lambda_{ab}}{[\tilde\lambda\pi]_a [\tilde\lambda \pi]_b}\Bigr)\,.
 \end{split}
\end{equation}
On the other hand, by Lemma~\ref{lema-merging}  we have that  
\begin{equation}\label{equasub245}
\begin{split}
\limsup_{n\to\infty} \; \frac12\sum_{1\leq a, b\leq k}& \frac{d^t_a(x,\widehat z_k)d^t_b(x,\widehat z_k)}{\rho_n^2n^4}\;\varphi\Bigl(\frac{ \rho_nn^2 \tilde o_{ab}(x,\widehat z_k)}{d^t_a(x,\widehat z_k)d^t_b(x,\widehat z_k)}\Bigr) \\
 &\leq \;\frac12 \sum_{1\leq a,  b\leq k} \pi^*_a\pi^*_b [\lambda^*\pi^*]_a [\lambda^*\pi^*]_b\;\varphi\Bigl(\frac{\lambda^*_{ab}}{ [\lambda^*\pi^*]_a [\lambda^*\pi^*]_b}\Bigr)
\end{split}
\end{equation}
for some $k\times k$ positive matrix $\lambda^*$  and $k$ dimensional vector $\pi^*$ defined by \eqref{lambda_star}. 
Finally, by Lemma~\ref{lema-merging2} we have that the difference of \eqref{equasub9} and \eqref{equasub245} 
 is lower bounded by 
 \[
\frac12 \Biggl(\sum_{1\leq a,  b\leq k_0}\pi_a\pi_b[\tilde\lambda\pi]_a [\tilde\lambda \pi]_b\;\varphi\Bigl(\frac{\tilde\lambda_{ab}}{[\tilde\lambda\pi]_a [\tilde\lambda \pi]_b}\Bigr)
 - \sum_{1\leq a, b\leq k} \pi^*_a\pi^*_b [\lambda^*\pi^*]_a [\lambda^*\pi^*]_b\;\varphi\Bigl(\frac{\lambda^*_{ab}}{ [\lambda^*\pi^*]_a [\lambda^*\pi^*]_b}\Bigr)\Biggr) \;>\;0
\]
unless $\tilde\lambda$ has two proportional columns, which contradicts the hypothesis of identifiability of $k_0$. This concludes the proof of Proposition~\ref{propd3}.
\end{proof}


We conclude the proof of Theorem~\ref{teod1} by proving that $\widehat{k}_n(x)$ does not overestimates $k_0$, the true number of communities. 

\begin{prop}\label{propos2}
For the DCSBM, the estimator $\widehat{k}_{n}$ defined in \eqref{eqd5} satisfies 
\[
\widehat{k}_n\;\leq\; k_0
\]
eventually almost surely as $n\to\infty$.
\end{prop}

{Observe that there is no assumption on $\rho_n$ for this proposition.}

\begin{proof}
By the  Borel-Cantelli Lemma, it is enough to prove that the following series converges
\begin{equation}\label{eq:prop_over1}
\begin{split}
\sum_{n=1}^\infty\sum_{k=k_0+1}^n p(\widehat{k}_n=k)\;=&\;\sum_{n=1}^\infty\sum_{k=k_0+1}^n\sum_{{  x}\in \Omega_n^c} p(x)\mathds{1}\lbrace\widehat{k}_n(x)=k\rbrace\\
&+\;\sum_{n=1}^\infty\sum_{k=k_0+1}^n\sum_{x\in \Omega_n} p(x)\mathds{1}\lbrace\widehat{k}_n(x)=k\rbrace.
\end{split}
\end{equation}
Let us start by the first term in the right-hand side, and observe that for each $n$
\[
\sum_{k=k_0+1}^n\sum_{x\in \Omega_n^c} p(x)\mathds{1}\lbrace\widehat{k}_n(x)=k\rbrace\;=\;\sum_{x\in \Omega_n^c} p(x)\sum_{k=k_0}^n\mathds{1}\lbrace\widehat{k}_n(x)=k\rbrace\;\leq\; \sum_{x\in \Omega_n^c}p(x)\;=\;p(\Omega_n^c).
\]
Now  using Lemma~\ref{lema:concentracao} we conclude that the first term of \eqref{eq:prop_over1} is indeed summable in $n$. 
So we now need to prove that the second term of \eqref{eq:prop_over1} is also summable in $n$. First observe that for some fixed $k=k_0+1,\dots,n$ we have that
\begin{equation}\label{lemm1i}
\sum_{x\in\Omega_n} p(x)\mathds{1}\lbrace\widehat{k}_n(x)=k\rbrace\;=\;\sum_{x\in\Omega_n} p(x)\mathds{1}\lbrace\underset{\ell}{\arg\max} ( \log p_\ell(x)-(\ell^3+3\ell n)\log (n+1))=k\rbrace.
\end{equation}
By the definition of $\widehat{k}_n(x)$, we have, when $k>k_0$, that 
\begin{equation*}
\begin{split}
\Bigl\{ \underset{\ell}{\arg\max}[\log & \, \log p_\ell(x)- (\ell^3 +3\ell n)\log(n+1)]=k\Bigr\}\\
&\subset \; \Bigl\{ \log p_k(x)-(k^3 +3kn)\log(n+1) \geq \log  p_{k_0}(x)- (k_0^3+3k_0n)\log(n+1)\Bigr\}\,.
\end{split}
\end{equation*}
So
\begin{equation}\label{lemm2i}
\begin{split}
\sum_{x\in\Omega_n} &p(x)\mathds{1}\lbrace\widehat{k}_n(x)=k\rbrace\\
&\;\leq\; \sum_{x\in\Omega_n}p(x)\mathds{1}\bigl\{\log p_k(x)-(k^3+3kn)\log(n+1) \geq \log  p_{k_0}(x)- (k_0^3+3k_0n)\log(n+1)\bigr\}\\
&=\sum_{x\in \Omega_n} p(x)\mathds{1}\bigl\{ p_{k_0}(x)\leq p_k(x)\exp[(-k^3-3kn+k_0^3+3k_0n)\log(n+1)]\bigr\}.
\end{split}
\end{equation}
Using Proposition~\ref{propd1}, we have for ${x}\in \Omega_n$ that 
\begin{align}\label{lemm3i}
\log p(x)&\; \leq\; \log\sup_{\theta\in \Theta^{k_0}} p(x|\theta).\nonumber\\
&\;\leq\; \log p_{k_0}(x)+ k_0(k_0+2)\log (n+1) + 3 n \log n,
\end{align}
giving that 
\begin{align}\label{lemm4i}
p(x)&\; \leq\;  p_{k_0}(x)e^{k_0(k_0+2)\log (n+1) + 3 n \log n}.
\end{align}
Then, on the set $\bigl\{ p_{k_0}(x)\leq p_k(x)\exp[(k_0^3+3k_0n-k^3-3kn)\log(n+1)]\bigr\}$,
we have that each term in the sum in \eqref{lemm2i} can be upper bound by 
\begin{equation}\label{lemm2ib}
\begin{split}
 p(x)\mathds{1}\bigl\{ p_{k_0}(x)&\leq p_k(x)e^{(k_0^3+3k_0n-k^3-3kn)\log(n+1)}\bigr\}\\[2mm]
 &\leq \;p_{k}(x)e^{k_0(k_0+2)\log(n+1) + 3 n \log n + (k_0^3+3k_0n-k^3-3kn)\log(n+1)}.
\end{split}
\end{equation}
Observe that as $k\geq k_0+1$, the exponent in \eqref{lemm2ib} can be upper bounded by 
\[
(-2k_0^2-k_0-1)\log (n+1) \;\leq\; - 4 \log n\,.
\]
Substituting now \eqref{lemm2ib} in \eqref{lemm2i} and summing in $k=k_0+1,\dots, n$ gives that 
\begin{equation}\label{lemm5i}
\begin{split}
\sum_{k=1}^n\sum_{x\in \Omega_n} p(x)\mathds{1}\bigl\{\widehat{k}_n(x)=k\bigr\} \; &\leq\;  n\,n^{-4} \sum_{x\in \Omega_n} p_{k}(x)\\
&\leq\; n^{-3}
\end{split}
\end{equation}
that is summable in $n$. This concludes the proof of Proposition~\ref{propos2}.
\end{proof}

\section*{Acknowledgments}
This work was produced as part of the activities of the \emph{Research, Innovation and Dissemination Center 
for Neuromathematics} (grant FAPESP 2013/07699-0). It was also supported by FAPESP project (grant 2017/10555-0)
\emph{``Stochastic Modeling of Interacting Systems''}  {and CNPq Universal project (grant 432310/2018-5) \emph{``Statistics, stochastic processes and discrete structures''}}.  FL is partially supported by a CNPq’s research fellowship, grant 311763/2020-0. During the realization of this work, CV was supported by a CAPES Ph.D. fellowship. 

\bibliographystyle{plainnat}
\bibliography{bibliography.bib}

\newpage
%

\section{Appendix}

\subsection{Basic results}

We state below Lemmas~\ref{lem:desigualdade_gamma} and \ref{lemma-basico-gamma} for completeness. The proofs are included in the Supplementary Material.

\begin{lemma}\label{lem:desigualdade_gamma}
For integers  $m=m_1+\cdots+m_J$ we have that
\begin{equation}
\dfrac{\prod\limits_{j=1}^J \left( \frac{m_j}{m} \right)^{m_j} }{\prod\limits_{j=1}^J  \Gamma\left(m_j + \frac{1}{2} \right)} \;\leq \;\dfrac{1}{\Gamma\left(m+ \frac{1}{2} \right)\Gamma\left( \frac{1}{2} \right)^{J-1}}\,.
\end{equation}
\end{lemma}

\begin{proof}
For an integer $m$, we have that
\[
\Gamma\left(m+\dfrac{1}2\right)=(m+1)(m+2)\dots(2m)\frac{\sqrt{\pi}}{2^m}\,.
\]
Thus, for integers $m_j$, $j=1,\dots,J$, such that $m=m_1+\cdots+m_J$ we write 
\begin{equation}\label{eq:lema_gamma1}
\dfrac{\prod\limits_{j=1}^J  \Gamma\left(  m_j + \frac{1}{2} \right)}{\Gamma\left(  m+ \frac{1}{2} \right)\Gamma\left( \frac{1}{2} \right)^{J-1}}= \dfrac{\prod\limits_{j=1}^J (m_j+1)(m_j+2)\cdots(2m_j) }{(m+1)(m+2)\cdots(2m)}\,.
\end{equation}
Define, for $y \geq 0$ and an integer $r\geq 1$ 
\[
g_r(y) = \prod\limits_{i=1}^r \left( y + \frac{i}{r}  \right)\,.
\]
As it is shown in the Lemma included in  the Appendix of \cite{davisson1981efficient}, for integers 
$m=m_1+\cdots+m_J$ and $y\geq 0$ we have that 
\begin{equation}
g_m(y) \; \leq\;  \prod\limits_{j=1}^Jg_{m_j}(y) \,.
\end{equation}
Using this result for $y=1$  and $r=m$ we have that
\begin{equation}\label{eq:lema_gamma2}
\begin{split}
\dfrac{\prod\limits_{i=1}^m \left( m + i  \right)}{m^m} &\;\leq\; \prod\limits_{j=1}^J\,\prod\limits_{i=1}^{m_j}\dfrac{ (m_j+i)}{m_j} =\;\dfrac{\prod\limits_{j=1}^J (m_j+1)(m_j+2)\cdots(2m_j) }{\prod\limits_{j=1}^J m_j^{m_j}}\,.
\end{split}
\end{equation}
\vspace{-0.5cm}
Rearranging \eqref{eq:lema_gamma2} and combining with \eqref{eq:lema_gamma1} we conclude that\\[2mm]
\begin{equation*}
\prod\limits_{j=1}^J \left( \frac{m_j}{m} \right)^{m_j} \;\leq\;\dfrac{\prod\limits_{j=1}^J (m_j+1)(m_j+2)\cdots(2m_j) }{(m+1)(m+2)\cdots(2m) }=\dfrac{\prod\limits_{j=1}^J  \Gamma\left(  m_j + \frac{1}{2} \right)}{\Gamma\left(  m+ \frac{1}{2} \right)\Gamma\left( \frac{1}{2} \right)^{J-1}}\,.\qedhere
\end{equation*}
\end{proof}

\begin{lemma}\label{lemma-basico-gamma}
For integers  $m=m_1+\cdots+m_J$, with $J\geq 1$ and $m\geq \max(J,3)$,  we have that 
\[
\frac{\Gamma\left( 1\over 2\right)\Gamma\left( m+{J\over 2}\right)}{\Gamma\left( J\over 2\right)\Gamma\left(m+{1\over 2}\right)} \;\leq\; m^{J}\,.
\]
\end{lemma}

\begin{proof}
Stirlings' formula for the $\Gamma$ function states that for all $y\geq 0$ we have
\[
y^{y-\frac12} e^{-y} \sqrt{2\pi} \;\leq\; \Gamma(y) \;\leq\; y^{y-\frac12}e^{-y} \sqrt{2\pi} e^{\frac1{12 y}}\,.
\]
Then
\begin{equation}\label{log_gamma}
\begin{split}
\log \left( \dfrac{\Gamma\left(\frac{1}{2}\right)\Gamma\left(m+\frac{J}{2}\right)}{\Gamma\left(\frac{J}{2}\right)\Gamma\left(m+\frac{1}{2}\right)}\right) &\;\leq\; \left(m + \frac{J-1}{2} \right)\log \left( m + \frac{J}{2}  \right) - \left( m + \frac{J}{2}  \right) + \dfrac{1}{12\left( m + \frac{J}{2}  \right)} \\
&\hspace{2cm}-m\log \left( m + \frac{1}{2}  \right) +\left( m + \frac{1}{2}  \right) + \log\dfrac{\Gamma(\frac{1}{2})}{\Gamma(\frac{J}{2})}\\
&\;\leq\; \left(m + \frac{J-1}{2} \right)\log \left( m \left(1 + \frac{J}{2m} \right)  \right) + \dfrac{1}{12m } -m\log \left( m \left(1 + \frac{1}{2m} \right) \right) \\
&\hspace{2cm}- \dfrac{J-1}{2}+ \log\dfrac{\Gamma(\frac{1}{2})}{\Gamma(\frac{J}{2})} \\
&\;\leq\; \left(\frac{J-1}{2} \right)\log m +\left( m + \frac{J-1}{2}  \right)\log \left( 1 + \frac{J}{2m}  \right) -m\log \left( 1 + \frac{1}{2m}  \right)  \\
&\hspace{2cm} + \dfrac{1}{12m} - \dfrac{J-1}{2}+ \log\dfrac{\Gamma(\frac{1}{2})}{\Gamma(\frac{J}{2})}\,. 
\end{split}
\end{equation}
Using that $1- \frac{1}{y}\leq \log y \leq y-1$, for  $y > 0$ we obtain that  
\begin{equation*}
\begin{split}
\log \left( \dfrac{\Gamma\left(\frac{1}{2}\right)\Gamma\left(m+\frac{J}{2}\right)}{\Gamma\left(\frac{J}{2}\right)\Gamma\left(m+\frac{1}{2}\right)}\right) & \leq \left(\frac{J-1}{2} \right)\log m +\left( m + \frac{J-1}{2}  \right) \left( \frac{J}{2m}  \right) +  \dfrac{1}{12m}- \dfrac{J-1}{2} +  \log\dfrac{\Gamma(\frac{1}{2})}{\Gamma(\frac{J}{2})}\\
&\leq \left(\frac{J-1}{2} \right)\log m + \dfrac{J(J-1)}{4m}+ \dfrac{1}{12m}+ \log\dfrac{\Gamma(\frac{1}{2})}{\Gamma(\frac{J}{2})}\,.
\end{split}
\end{equation*}
Observe that for $J\geq 1$ and  $m\geq \max(J,3)$ we have that  
\[
 \log\dfrac{\Gamma(\frac{1}{2})}{\Gamma(\frac{J}{2})} \;\leq  \;\log(2) 
\]
and
\[
\frac{J(J-1)}{4m} \;\leq\;  \frac{J-1}{4} \;\leq\;  \frac{J-1}{4} \log m\,.
\]
Then 
\begin{equation*}
\begin{split}
\log \left( \dfrac{\Gamma\left(\frac{1}{2}\right)\Gamma\left(m+\frac{J}{2}\right)}{\Gamma\left(\frac{J}{2}\right)\Gamma\left(m+\frac{1}{2}\right)}\right) & \;\leq\; \left(\frac{J-1}{2} \right)\log m +  \frac{J-1}{4} \log m + 1\leq\; J\log m\,.\qedhere
\end{split}
\end{equation*}
\end{proof}

\subsection{Proof of Proposition~\ref{propd1}}
For each $z\in [k]^n$, the space of hiperparameters $\Theta_k=\Theta_k(z)$  has the form 
\[
 \Theta_k =\Pi_k\times\Lambda_k\times\mathcal W(z)
\]
where 
\[
\Pi_k:=\{(\pi_1,\ldots, \pi_k)\in{(0,1]}^{k}\colon \sum_i \pi_i=1\}
\]
is the standard $k$-dimensional simplex, 
\[
\Lambda_k:=\{\lambda\in(\mathbb{R}^{+})^{k\times k}\colon \lambda_{ab}=\lambda_{ba} \text{ for all }a,b\in [k]\}
\]
is the set of $k\times k$ symmetric matrices with positive entries and 
\[
\mathcal W(z) = \mathcal W_1(z)\times \mathcal W_2(z)\times  \dots \times\mathcal W_k(z)
\]
with 
\begin{equation}\label{eq:wworld}
\mathcal W_a(z):=\{w\in\mathbb (\mathbb R^+)^{n_a(z)}:\sum_{i}w_i= n_a(z)\}
\end{equation}
which is the set of possible $w$'s on community $a$.  
By the definition of the model we have that the \emph{a priori}  distribution  over $\Theta_k$ is 
given by 
\begin{equation}\label{eq:priori}
\nu_k(\theta)=\nu_k^{(1)}(\pi)\nu_k^{(2)}(\lambda)\nu^{(3)}(w|z)
\end{equation} 
where 
\[
\nu_k^{(1)}(\pi) = \dfrac{\Gamma\left( k/2\right)}{{\Gamma\left(1/2\right)}^{k}}\prod_{1\leq a\leq k}{\pi_a}^{-1/2}\,,
\]
\[
\nu_k^{(2)}(\lambda) = \dfrac{1}{
\Gamma(1/2)^{\frac{k(k+1)}2}} \prod_{1\leq a\leq b\leq k}\,\lambda_{ab}^{-1/2}e^{-\lambda_{ab}}
\]
and
\[
\nu^{(3)}(w|z) = \prod_{1\leq a\leq k} \,
\frac{\Gamma\left({{n_a(z)}\over 2}\right)}{{n_a(z)\Gamma\left({1\over 2}\right)}^{n_a(z)}} \prod_{i\colon z_i=a} \Bigl(\frac{w_i}{n_a(z)}\Bigr)^{-1/2}\,.
\]
We can  now decompose the marginal likelihood as
\begin{align*}
p_k(x) &\;=\; \sum_{z\in [k]^n} \int_{{\Theta}^{k}} p(x,z|\theta)\nu_k(\theta)d\theta\\
& \;=\; \sum_{z\in [k]^n}\int_{\mathcal W(z)} \int_{\Lambda^k} p(x|w,z,\lambda) \nu^{(2)}_k(\lambda) \nu^{(3)}(w|z)d\lambda dw \int_{\Delta^k} p(z|\pi)\nu^{(1)}_k(\pi)d\pi\\
&\;=:\;\sum_{z\in [k]^n}  p_k(x|z)p_k(z)
\end{align*}
in which the (conditional) likelihoods were given in \eqref{eq:z} and \eqref{eq:x|w,z}. Then we have that
\begin{equation}\label{eq:defAPIL}
\begin{split}
 \int_{\Lambda^k}  p(x|w,z,\lambda) &\nu^{(2)}_k(\lambda)d\lambda\;=\;\int_{\Lambda^k}\dfrac{1}{c(x)}\Biggl[\prod_{1\leq i\leq n}{w_i}^{d_i(x)}\Biggr]\Biggl[ \prod_{1\leq a\leq b\leq k}\lambda_{ab}^{o_{ab}(x,z)}e^{-n_{ab}(z) \lambda_{ab}}\Biggr]\nu^{(2)}_k(\lambda)d\lambda\\
&=\dfrac{1}{c(x)\Gamma({1\over 2})^{k(k+1)\over 2}}\Biggl[\prod_{1\leq i\leq n}{w_i}^{d_i(x)}\Biggr]
\int_{\Lambda^k} \prod_{1\leq a\leq b\leq k}\lambda_{ab}^{o_{ab}{(x,z)}-\frac12}e^{-(n_{ab}(z)+1)\lambda_{ab}}d\lambda\\
&=\dfrac{1}{c(x)}\underbrace{\dfrac{1}{\Gamma({1\over 2})^{k(k+1)\over 2}}\Biggl[\prod_{1\leq a\leq b\leq k} \frac{\Gamma\left( o_{ab}(x,z)+{1\over 2}\right)}{[n_{ab}(z)+1]^{o_{ab}(x,z)+\frac12}} \Biggr]}_{A(x,z)}\Biggl[\prod_{1\leq i\leq n}{w_i}^{d_i(x)}\Biggr].
\end{split}
\end{equation}
Therefore
\begin{equation}
\begin{split}
p_k(x|z) &= \frac{A(x,z)}{c(x)} \int_{\mathcal W} \prod_{1\leq i\leq n}{w_i}^{d_i(x)} \nu^{(3)}(w|z)dw \\
&= \frac{A(x,z)}{c(x)}
\prod_{1\leq a\leq k}
\frac{\Gamma\left({{n_a(z)}\over 2}\right)}{{n_a(z)\Gamma\left({1\over 2}\right)}^{n_a(z)}} 
\int_{\mathcal W_a(z)} \prod_{i\colon z_i=a}{w_i}^{d_i(x)}\Bigl(\frac{w_i}{n_a(z)}\Bigr)^{-1/2}  dw_a\\
&=\frac{A(x,z)}{c(x)}\prod_{1\leq a\leq k}
\frac{\Gamma\left({{n_a(z)}\over 2}\right)n_a(z)^{d^t_a(x,z)-1}}{{\Gamma\left({1\over 2}\right)}^{n_a(z)}} 
\int_{\mathcal W_a(z)} \prod_{i\colon z_i=a}{ \Bigl(\frac{w_i}{n_a(z)}\Bigr)}^{d_i(x)-1/2} dw_a\\
&=\frac{A(x,z)}{c(x)}\prod_{1\leq a\leq k}
\frac{\Gamma\left({{n_a(z)}\over 2}\right)n_a(z)^{d^t_a(x,z)}}{{\Gamma\left({1\over 2}\right)}^{n_a(z)}} 
\int_{\mathcal Y_a(z)} \prod_{i\colon z_i=a}{ y_i}^{d_i(x)-1/2} dy\\
&=\frac{A(x,z)}{c(x)}\underbrace{\prod_{1\leq a\leq k} \frac{n_a(z)^{d^t_a(x,z)}\Gamma\left({{n_a(z)}\over 2}\right)}{{\Gamma\left({1\over 2}\right)}^{n_a(z)}}
\frac{\prod_{i\colon z_i=a}\Gamma(d_i(x)+\frac12)}{\Gamma(d^t_a(x,z)+\frac{n_a(z)}2)}}_{B(x,z)}\,.
\end{split}
\end{equation}
We also have that 
 \begin{equation}
 \begin{split}
p_k(z)  &=\; \int_{\Pi^k} \prod_{a=1}^k {\pi_a}^{n_a({z})}\nu^{(1)}_k(\pi)d\pi\\
&=\;\dfrac{\Gamma\left( {k\over 2}\right)}{{\Gamma\left({1\over 2}\right)}^{k}}\int_{\Pi^k} \prod_{a=1}^k {\pi_a}^{n_a({z})-{1\over 2}}d{\pi} \\
&=\;\underbrace{\dfrac{\Gamma\left( {k\over 2}\right)}{{\Gamma\left({1\over 2}\right)}^{k}}\dfrac{\prod_{a=1}^k\Gamma\left(n_a(z)+{1\over 2}\right)}{\Gamma\left(n+{k\over 2}\right)}}_{C(z)}\,.\label{equa5}
\end{split}
\end{equation}
On the other hand, by the definition of the maximum likelihood estimators we have that 
\begin{equation}\label{equ3}
\begin{split}
\sup_{\lambda,w} \; p(x|z,w,\lambda) 
&= {1\over c(x)} \underbrace{ \Bigl[\prod_{1\leq a\leq b\leq k}\left( \frac{o_{ab}(x,z)}{n_{ab}(z)}\right)^{o_{ab}(x,z)}e^{- o_{ab}(x,z)}\Bigr]}_{\widehat A(x,z)}\\
&\qquad\times \underbrace{\Bigl[\prod_{a\in[k],i\in[n]:z_i=a}\left(  \frac{n_a(z) d_i(x)}{d^t_a(x,z)}\right)^{d_i(x)}\Bigr]}_{\widehat B(x,z)}\,
\end{split}
\end{equation}
and
\begin{equation}\label{equ4}
\begin{split}
\sup_{\pi} \; p(z|\pi)
&= \underbrace{\prod_{1\leq a\leq k_0} \left(  \frac{n_a(z)}{n}\right)^{n_a(z)}}_{\widehat C(z)}  \,.
\end{split}
\end{equation}
Now, we observe that by canceling the normalizing constant $c(x)$ we obtain that 
\begin{align}
\frac{\sup_\theta p(x|\theta)}{p_k(x)}&\;\leq\; 
 \frac{ \sum_z \widehat A(x,z)\widehat B(x,z)\widehat C(z)}{\sum_zA(x,z) B(x,z) C(z)} \,.
\end{align}
Now, if we are able to find  bounds $D_1, D_2$  and $D_3$, uniform on $x$ and $z$, such that
\[
\frac{\widehat A(x,z)}{A(x,z)}\;\leq\; D_1\,,\quad \frac{\widehat B(x,z)}{B(x,z)}\;\leq\; D_2\,\quad \text{and}\quad 
\frac{\widehat C(x,z)}{C(x,z)}\;\leq\;  D_3
\]
then we automatically get
\begin{equation}\label{eq:upperboundration}
\dfrac{\sup_\theta p(x|\theta)}{p_k(x)}\;\leq\; D_1\,D_2\,D_3\,.
\end{equation}
These bounds follow by Lemmas~\ref{lem:ratio1}, \ref{lem:ratio2} and \ref{lem:ratio3} proved below.
Using these lemmas  we obtain that 
\begin{equation*}
\begin{split}
\log \frac{\sup_\theta p(x|\theta)}{ p_k(x)}\;&\le\; \log \Bigl[(n+1)^{k(k+1)}(n^2\log n)^{n}n^{k}\Bigr]\\
&\leq\; k(k+1)\log(n+1)+n(2\log n + \log\log n)+ k\log n\\
&\leq\; k(k+2)\log (n+1) + 3n\log n
\end{split}
\end{equation*}
concluding the proof of Proposition~\ref{propd1}.

\begin{lemma}\label{lem:ratio1}
For $(x,z)\in \Omega_n\times [k]^n$ we have that 
\begin{equation}\label{equ11}
\frac{\widehat A(x,z)}{A(x,z)}\;\leq\; (n+1)^{k(k+1)}\,.
\end{equation}
\end{lemma}
 \begin{proof}
Fix any $(x,z)\in \Omega_n\times [k]^n$. We simplify the notation by writing
$o_{ab}=o_{ab}(x,z)$ and similarly for $n_a(z)$ and $n_{ab}(z)$. 
By  \eqref{eq:defAPIL} an \eqref{equ3} we obtain that 
\begin{align}
\frac{\widehat A(x,z)}{A(x,z)}&\;=\; \dfrac{\prod_{1\leq a\leq b\leq k}\left( \frac{o_{ab}}{n_{ab}}\right)^{o_{ab}}e^{-o_{ab}}}{\dfrac{1}{\Gamma({1\over 2})^{k(k+1)\over 2}}\prod_{1\leq a\leq b\leq k} \frac{\Gamma\left( o_{ab}+{1\over 2}\right)}{[n_{ab}+1]^{o_{ab}+\frac12}}}\nonumber\\
&=\Gamma\left({1\over 2}\right)^{k(k+1)\over 2}\prod_{1\leq a\leq b\leq k}\dfrac{\left({o_{ab}}\over{n_{ab}}\right)^{o_{ab}}e^{- o_{ab}}{\left(n_{ab}+1\right)^{o_{ab}+\frac12}}}{{\Gamma\left({o_{ab}}+{1\over 2}\right)}}.\label{equ4}
\end{align}
Letting 
\begin{equation*}
N :=\sum_{1\leq a\le b\leq k} o_{ab} 
\end{equation*}
we rewrite 
\begin{equation}\label{equ5}
\frac{\widehat A(x,z)}{A(x,z)}=\Gamma\left({1\over 2}\right)^{k(k+1)\over 2}\prod_{1\leq a\le b\leq k}\dfrac{\left({o_{ab} \over N }\right)^{o_{ab} }}{\Gamma\left(o_{ab} +{1\over 2}\right)}\left({N \over n_{ab}  }\right)^{o_{ab} }\left(n_{ab} +1\right)^{o_{ab} +{1\over 2}}e^{-o_{ab} }.
\end{equation}
We now use Lemma~\ref{lem:desigualdade_gamma} to get
\begin{small}
\begin{equation}\label{equ6}
\prod_{1\leq a\leq b\leq k}\dfrac{\left({o_{ab} \over N }\right)^{o_{ab} }}{\Gamma\left({o_{ab} }+{1\over 2}\right)}\;\leq\; \dfrac{1}{\Gamma\left(N +{1\over 2}\right)\Gamma\left({1\over 2}\right)^{{k(k+1)\over 2}-1}}\,.
\end{equation}
\end{small}
On the other hand
\begin{align}\label{equ7}
\left({N \over n_{ab} }\right)^{o_{ab} }\left(n_{ab} +1\right)^{o_{ab} +{1\over 2}}e^{-o_{ab} }&\;=\;
N^{o_{ab}}\left(1+ \frac{1}{n_{ab}}\right)^{o_{ab}}\left(n_{ab}+1\right)^{1\over 2}e^{-o_{ab}}\nonumber\\
&\;\leq\;  N^{o_{ab}} \; e^{o_{ab}\over n_{ab}}\left(n_{ab}+1\right)^{1\over 2}e^{- o_{ab}}.
\end{align}
Putting \eqref{equ6}  and \eqref{equ7} in \eqref{equ5}, we get
\begin{align}\label{equ8}
\frac{\widehat A(x,z)}{A(x,z)}&\;\leq\; \frac{\Gamma\left({1\over 2}\right)^{k(k+1)\over 2}}{\Gamma\left(N+{1\over 2}\right)\Gamma\left({1\over 2}\right)^{{k(k+1)\over 2}-1}}\prod_{1\leq a\leq b\leq k} N^{o_{ab}}e^{\frac{o_{ab}}{n_{ab}}} \left(n_{ab}+1\right)^{1\over 2}e^{-o_{ab}}\nonumber \\
&=\frac{\Gamma\left({1\over 2}\right)}{\Gamma\left(N+{1\over 2}\right)}e^{-N}N^{N}e^{\sum_{1\leq a\leq b\leq k}{o_{ab}\over n_{ab}}}\prod_{1\leq a\leq b\leq k}\left(n_{ab}+1\right)^{1\over 2}.
\end{align}
For any real number $r>0$, $\Gamma(r)\geq r^{r-1/2}e^{-r}\sqrt{2\pi}$, so that
 \begin{equation}\label{equ9}
\dfrac{e^{-N}N^{N}}{\Gamma\left(N+{1\over 2}\right)}\;\leq\; \left(1+{1\over 2N}\right)^{-N}e^{1\over2}\frac{1}{\sqrt{2\pi}}\;\leq\; \frac{1}{\sqrt{2\pi}}.
\end{equation}
 Moreover, as $x\in\Omega_n$ we have that 
 \begin{align}\label{equ10}
\sum_{1\leq a\leq b\leq k}{o_{ab} \over n_{ab}}\;\leq\; \sum_{1\leq a\leq b\leq k}{n_{ab} \log n  \over n_{ab}}\;=\; {k(k+1)\over 2} \log n
\end{align}
and
\begin{equation}\label{equ11}
\prod_{1\leq a\leq b\leq k}\left(n_{ab}+1\right)^{1\over 2} \;\leq\; \prod_{1\leq a\leq b\leq k} \left(n^2+1\right)^{1\over 2}\;=\;\left(n^2+1\right)^{k(k+1)\over 4}\;\leq\; (n+1)^{k(k+1)\over 2}.
\end{equation}
Plugging \eqref{equ9}, \eqref{equ10} and \eqref{equ11} into  \eqref{equ8} proves the lemma. 
\end{proof}

\begin{lemma}\label{lem:ratio2}
For $(x,z)\in \Omega_n\times [k]^n$ we have 
\begin{equation}
\frac{\widehat B(x,z)}{B(x,z)}\;\leq\; (n^2\log n)^{n}.
\end{equation}
\end{lemma}
\begin{proof}
For $(x,z)\in \Omega_n\times [k]^n$ we have that 
\begin{align*}
\frac{\widehat B(x,z)}{B(x,z)}&= \prod_{1\leq a\leq k}
 \frac{\prod_{i\in[n]:z_i=a}\left(  \frac{n_a(z) d_i(x)}{d^t_a(x,z)}\right)^{d_i(x)}}{ \frac{n_a(z)^{d^t_a(x,z)}\Gamma\left({{n_a(z)}\over 2}\right)}{{\Gamma\left({1\over 2}\right)}^{n_a(z)}}
\frac{\prod_{i\colon z_i=a}\Gamma(d_i(x)+\frac12)}{\Gamma(d^t_a(x,z)+\frac{n_a(z)}2)}}\\
&=\prod_{1\leq a\leq k} \frac{{\Gamma\left({1\over 2}\right)}^{n_a(z)}{\Gamma(d^t_a(x,z)+\frac{n_a(z)}2)}}{\Gamma\left({{n_a(z)}\over 2}\right)}    
\prod_{i:z_i=a} \frac{\left(  \frac{d_i(x)}{d^t_a(x,z)}\right)^{d_i(x)}}
{\Gamma(d_i(x)+\frac12)}
\end{align*}
But for all $a\in[k]$ we have by Lemma~\ref{lem:desigualdade_gamma} that 
\begin{equation}\label{equa8}
\prod_{i:z_i=a} \frac{\left(  \frac{d_i(x)}{d^t_a(x,z)}\right)^{d_i(x)}}
{\Gamma(d_i(x)+\frac12)}
\;\leq\; \dfrac{1}{\Gamma\left(d^t_a(x,z)+{1\over 2}\right){\Gamma\left(1\over 2\right)}^{n_a(z)-1}},
\end{equation}
where we used the equality $d^t_a(x,z)=\sum_{i:z_i=a} d_i(x)$.
Putting all the previous bounds together we obtain that 
\begin{align*}
\frac{\widehat B(x,z)}{B(x,z)}&\;\leq \;\prod_{1\leq a\leq k} \frac{{\Gamma\left({1\over 2}\right)}{\Gamma(d^t_a(x,z)+\frac{n_a(z)}2)}}{\Gamma\left({{n_a(z)}\over 2}\right)\Gamma\left(d^t_a(x,z)+{1\over 2}\right)}.    
\end{align*}
Finally, by Lemma~\ref{lemma-basico-gamma} 
 we conclude, as  $d^t_a(x,z) \leq n^2\log n$ for all $a\in[k]$ and $\sum_a n_a(z)=n$   that
\begin{align*}
\frac{\widehat B(x,z)}{B(x,z)}\;\leq\;  \prod_{1\leq a\leq k}d^t_a(x,z)^{n_a(z)} \;\leq\;
 (n^2\log n)^{n}\,.
 \end{align*}
\end{proof}

\begin{lemma}\label{lem:ratio3}For any $z\in[k]^n$ we have that 
\begin{equation}
\frac{\widehat C(z)}{C(z)}\;\leq\; n^{k}\,.
\end{equation}
\end{lemma}

\begin{proof}
By definition we have that 
\begin{align*}
\frac{\widehat C(z)}{C(z)}&=
\frac{\prod_{a=1}^k {\left(\dfrac{n_a( {z})}{n}\right)}^{n_a( {z})}}
{\dfrac{\Gamma\left( {k\over 2}\right)}{{\Gamma\left({1\over 2}\right)}^{k}}\dfrac{\prod_{a=1}^k\Gamma\left(n_a(z)+{1\over 2}\right)}{\Gamma\left(n+{k\over 2}\right)}}\,.
\end{align*}
By Lemma~\ref{lem:desigualdade_gamma} we obtain that 
\begin{equation}\label{equa8}
\prod_{a=1}^k \dfrac{{\left(\dfrac{n_a(z)}{n}\right)}^{n_a(z)}}{\Gamma\left(n_a(z)+{1\over 2}\right)}\;\leq\;\dfrac{1}{\Gamma\left(n+{1\over 2}\right){\Gamma\left(1\over 2\right)}^{k-1}}\,.
\end{equation}
Then by Lemma~\ref{lemma-basico-gamma} and \eqref{equa8} we have  that
\begin{align*}
\frac{\widehat C(z)}{C(z)}\;\leq\; \frac{{{\Gamma\left({1\over 2}\right)\Gamma\left(n +{k\over 2}\right)}}}{{\Gamma\left( {k\over 2}\right)\Gamma\left(n+{1\over 2}\right)}}\;\leq\; n^{k}\,.
\end{align*}
\end{proof}

\subsection{Proof of other auxiliary lemmas}

\begin{lemma}\label{lema:concentracao}
Let $X=(X_{ij})_{i,j\in[n]}$ be generated by a DCSBM and let $\Omega_n$ be the set defined in \eqref{Omega}. 
Then $X\in\Omega_n$ eventually almost surely as $n\to\infty$.
\end{lemma}

\begin{proof}
We have that
\begin{equation}\label{equap11}
\begin{split}
p(\Omega_n^c)&\;=\;\sum_{z\in[k_0]^n}p(\Omega_n^c\cap \{Z=z\})\\
&\;=\;\sum_{z\in[k_0]^n} p(z)p(\Omega_n^c|Z=z)\\
&\leq\sum_{z\in[k_0]^n}p(z)\sum_{i,j=1}^{n} p(X_{ij}>\log n|Z=z).
\end{split}
\end{equation}
Conditionally on $Z=z$, the $X_{ij}$'s have a Poisson distribution with parameter $(\lambda_{z_iz_j})$. 
From \cite[Section 2.2]{boucheron2013concentration}, if $Y\sim\text{Poisson}(\mu)$, then for $r>0$
\begin{equation}\nonumber
\begin{split}
P(Y-\mu\ge r)&\le e^{-\mu \left(1+\frac{r}{\mu}\right)\log \left(1+\frac{r}{\mu}\right)+r}\\
\label{eq:poisson_max}&\le e^{-r\left(\log \left(1+\frac{r}{\mu}\right)-1\right)},
\end{split}
\end{equation}
thus for $n$ such that $\log n>\lambda_{z_iz_j}$
\begin{align}
\nonumber p(X_{ij}\ge \log n|Z=z)
\nonumber&\le e^{-(\log n-\lambda_{z_iz_j})\left(\log \left(\frac{\log n}{\lambda_{z_iz_j}}\right)-1\right)}\\
\label{eq:star}&\le e^{-(\log n-\lambda_{\max})\left(\log \left(\frac{\log n}{\lambda_{\max}}\right)-1\right)} 
\end{align}
where in the last line we used the notation $\lambda_{\max}=\max_{a,b}\lambda_{ab}$. 
Then, by \eqref{equap11} and \eqref{eq:star} we have that 
\begin{align*}
p(\Omega_n^c)\le\; n^2 e^{-(\log n-\lambda_{\max})\left(\log \left(\frac{\log n}{\lambda_{\max}}\right)-1\right)} \sum_{z\in[k_0]^n}p(z)=n^{3+\log\lambda_{\max}-\log(\log n)}\left(\log n\right)^{\lambda_{\max}}\left(e\lambda_{\max}\right)^{-\lambda_{\max}}
\end{align*}
which is summable in $n$ and therefore, by the Borel Cantelli Lemma we have that $X\in \Omega_n$ eventually almost surely as $n\to\infty$. 
\end{proof}

To state the next auxiliary result, we need some notation. Define for all $z\in\{1,\dots, k\}^n$  and  $z^0\in \{1,\dots, k_0\}^n$ the  $k\times k_0$ matrix $Q_n(z,z^0)$ given 
by 
\begin{equation}
[Q_n(z,z^0)]_{aa'} \;= \; {1\over n}\sum_{i=1}^n w_i\,\mathds{1}\lbrace z_i=a, z_i^0=a'\rbrace.
\end{equation}
Observe that the counters $n_{a'}(z^0)$, for $a'\in[k_0]$, can be written as 
\begin{equation}
n_{a'}(z^0)= \sum_{i=1}^n w_i\, \mathds{1}\lbrace z_i^0=a'\rbrace =  \sum_{i=1}^n w_i\, \sum_{a=1}^{k}\mathds{1}\lbrace z_i=a, z_i^0=a'\rbrace\,.
\end{equation}
Then $n_{a'}(z^0) =  n(Q_n^T(z,z^0){\bf 1}_{k})_{a'}$, with ${\bf 1}_{k}$  a column vector of dimension $k$
with all entries equal to 1.
Moreover, the matrix $Q_n(z,z^0)$ satisfies 
\begin{equation}
\| Q_n(z,z^0)\|_1 \;=\; \sum_{a=1}^k\sum_{a'=1}^{k_0}  [Q_n(z,z^0)]_{a,a'} = 1,\,
\end{equation}
for all $(z,z^0)$ and
\begin{equation}
\E( \tilde o_{ab}(x,z) \,|\, Z=z^0) \;=\; n^2Q_n(z,z^0) \lambda Q_n(z,z^0)^T \;=\; \rho_n n^2Q_n(z,z^0) \tilde\lambda Q_n(z,z^0)^T \,. 
\end{equation}

We now prove a concentration bound for $\tilde o_{ab}(x,z)$ conditionally on $Z=z^0$, that is  a sum of independent Poisson random variables. 

\begin{lemma}\label{lema-sandro}
For any $\epsilon>0$  and $a,b\in [k]$ we have that 
\[
\mathbb P\Bigl( \;\sup_{\bar z\in [k]^n}\; \Bigl|\frac{\tilde o_{ab}(X,\bar z)}{\rho_n n^2} - [Q_n(\bar z,Z)\tilde\lambda Q_n(\bar z,Z)^T]_{ab}\Bigr|>\epsilon\Bigr)\;\le\; \exp\Bigl( - \frac{\rho_nn^2\epsilon^2}{\tilde{\lambda}_{\max}+\epsilon}+n\log k\Bigr)
\]
and 
\[
\mathbb P\Bigl(\, \sup_{\bar z \in [k]^n}\, \Bigr|\frac{d^t_{a}(X,\bar z)}{\rho_n n^2} - [Q_n(\bar z,Z)\tilde\lambda Q_n(\bar z,Z)^T{\bf 1}_k]_{a}\Bigr|\;>\;\epsilon\Bigr) \; \leq\;  \exp\Bigl(-\frac{\rho_n n^2\epsilon^2}{\tilde{\lambda}_{\max}+\epsilon} + n\log k\Bigr)\,,
\]
with   $\tilde{\lambda}_{\max}=\max_{ab} \tilde \lambda_{ab}$. 
\end{lemma}

\begin{proof}
For any fixed $z\in [k_0]^n$ and $\bar z \in[k]^n$ we have that 
\begin{equation*}
\begin{split}
\tilde o_{ab}(X,\bar z) &-  n^2[Q_n(\bar z,z)\lambda Q_n(\bar z,z)^T]_{ab}\\
&= \tilde o_{ab}(X,\bar z) -  \rho_n n^2[Q_n(\bar z,z)\tilde \lambda Q_n(\bar z,z)^T]_{ab}\\
& = \sum_{1\leq i,j\leq n} \sum_{1\leq a',b'\leq k_0}  (X_{ij}-  \rho_n w_iw_j\tilde \lambda_{a'b'}) {\1}\lbrace\bar z_i=a, z_i=a'\rbrace {\1}\lbrace \bar z_j=b, z_j=b'\rbrace\,.
\end{split}
\end{equation*}
Observe that  given $Z=z$,  $\tilde o_{ab}(X,\bar z)$ corresponds to the sum of $n_a(\bar z)n_b(\bar z)$ independent Poisson random variables, given by $X_{ij}{\1}\lbrace\bar z_i=a, \bar z_i=b\rbrace$,  with expected value  given by $\rho_n w_i w_j\tilde\lambda_{z_iz_j}$.  
Then the sum is also Poisson distributed with a parameter that is the sum of the corresponding parameters. Using one more time  \cite[Section 2.2]{boucheron2013concentration}, we have for $Y\sim\text{Poisson}(\mu)$ and  $t>0$
\begin{align*}
P(Y-\mu\ge r)&\le  e^{-\mu \left(1+\frac{r}{\mu}\right)\log \left(1+\frac{r}{\mu}\right)+r}\\
P(Y-\mu\le -r)&\le  e^{-\mu \left(1-\frac{r}{\mu}\right)\log \left(1-\frac{r}{\mu}\right)-r}
\end{align*}
which, after some algebra, yields 
\[
\mathbb P(|X-\mu|>r )\; \leq\; 2 e^{-\frac{r^2}{2(\mu+r)}}\,.
\]
Therefore,  for any $\delta>0$
\begin{align*}
\mathbb P\Bigl(\,\Bigl|\tilde o_{ab}(X,\bar z) &- \rho_n n^2[Q_n(\bar z,z)\tilde\lambda Q_n(\bar z,z)^T]_{ab}\Bigr|\,>\,\delta \, | Z\,=z\Bigr)\\
&\leq \; 2\exp\Bigl( -\frac{\delta^2}{2(\rho_n n^2[Q_n(\bar z,z)\tilde \lambda Q_n(\bar z,z)^T]_{ab}+\delta)}\Bigr)\,.
\end{align*}
Since,  for any $\bar z$ and $z$, we have that 
\[
 \rho_n n^2[Q_n(\bar z,z)\tilde\lambda Q_n(\bar z,z)^T]_{ab}\;\leq\;  \rho_n n^2  \tilde{\lambda}_{\max}
\]
with  $\tilde{\lambda}_{\max}=\max_{ab} \tilde \lambda_{ab}$, it follows that, for any $z,\bar z$ and $\epsilon>0$
\begin{align*}
\mathbb P\Bigl(\, \Bigr|\frac{\tilde o_{ab}(X,\bar z)}{\rho_nn^2} - [Q_n(\bar z,z)\tilde \lambda Q_n(\bar z,z)^T]_{ab}\Bigr|\;>\;\epsilon\,|\, Z=z\Bigr) \;& \leq\;  \exp\Bigl(-\frac{\rho_n^2 n^4\epsilon^2}{\rho_nn^2\tilde{\lambda}_{\max}+\epsilon  \rho_n n^2}\Bigr)\\
& \leq\;  \exp\Bigl(-\frac{\rho_n n^2\epsilon^2}{\tilde{\lambda}_{\max}+\epsilon}\Bigr)\,.
\end{align*}
Now, using a union bound over all $\bar z\in [k]^n$ and integrating over $z$ we obtain that 
\[
\mathbb P\Bigl(\, \sup_{\bar z \in [k]^n}\, \Bigr|\frac{\tilde o_{ab}(X,\bar z)}{\rho_nn^2} - [Q_n(\bar z,Z)\tilde\lambda Q_n(\bar z,Z)^T]_{ab}\Bigr|\;>\;\epsilon\Bigr) \; \leq\;  \exp\Bigl(-\frac{\rho_n n^2\epsilon^2}{\tilde{\lambda}_{\max}+\epsilon} + n\log k\Bigr)
\]
and this proves the first inequality of the lemma.
Now,  given $Z=z$,
\[
d_a^t(X,\bar z) \;=\; \sum_{b\in [k]} \tilde o_{ab}(X,\bar z) 
\]
is also a sum of independent random variables with Poisson distribution 
and
\[
\E(d^t_a(X,\bar z)\, |\, Z=z) \;=\;  [\rho_n n^2Q_n(\bar z,z)\tilde\lambda Q_n(\bar z,z)^T{\bf 1}_k]_{a},
\]
thus we also obtain that 
\[
\mathbb P\Bigl(\, \sup_{\bar z \in [k]^n}\, \Bigr|\frac{d^t_{a}(X,\bar z)}{\rho_n n^2} - [Q_n(\bar z,Z)\tilde\lambda Q_n(\bar z,Z)^T{\bf 1}_k]_{a}\Bigr|\;>\;\epsilon\Bigr) \; \leq\;  \exp\Bigl(-\frac{\rho_n n^2\epsilon^2}{\tilde{\lambda}_{\max}+\epsilon} + n\log k\Bigr)\,.
\]
This concludes the proof of Lemma~\ref{lema-sandro}
\end{proof}

In the sequel, we state and prove the lemmas cited in the proof of Proposition~\ref{propd3}. 

\begin{lemma}\label{lema-merging}
For $k<k_0$  there exists a $k\times k$ positive matrix $\lambda^*$ and $k$ dimensional vector $\pi^*$ such that 
\begin{equation}
\begin{split}
\limsup_{n\to\infty} \; \frac12\sum_{1\leq a, b\leq k}& \frac{d^t_a(x,\widehat z_k)d^t_b(x,\widehat z_k)}{\rho_n^2n^4}\;\varphi\Bigl(\frac{ \rho_nn^2 \tilde o_{ab}(x,\widehat z_k)}{d^t_a(x,\widehat z_k)d^t_b(x,\widehat z_k)}\Bigr) \\
 &\leq \;\frac12 \sum_{1\leq a, b\leq k}  \pi^*_a\pi^*_b[\lambda^*\pi^*]_{a}[\lambda^*\pi^*]_{b} \; \varphi\biggl(\dfrac{\lambda^*_{ab}}{ [\lambda^*\pi^*]_{a}[\lambda^*\pi^*]_{b}}\biggr)\,.
 \end{split}
\end{equation}
Moreover,  $(\pi^*,\lambda^*)$ are given by 
\begin{equation}\label{lambda_star}
\begin{split}
	\pi^*_a&\;=\;[R^*{\bf 1}_{k_0}]_a, \,\qquad a\in \{1,\dots,k\}\\
	\lambda^*_{ab}&\;=\; \frac{[R^*\lambda{R^{*^{T}}}]_{ab}}{[R^*{\bf 1}_{k_0}{\bf 1}_{k_0}^T R^{*^T}]_{ab}}\,,\qquad a,b\in \{1,\dots,k\}\,,
\end{split}
\end{equation}
for  a $k\times k_0$ real matrix  $R^*$  satisfying  $\|R^*\|_1=1$ and having one and only one non-zero entry on each column. 
\end{lemma}
\begin{proof}
Observe that 
\begin{equation}\label{eq2302}
\begin{split}
\sum_{1\leq a,  b\leq k}&
\frac{d^t_a(x,\widehat z_k)d^t_b(x,\widehat z_k)}{\rho_n^2n^4}\varphi\Bigl(\frac{\rho_nn^2\tilde o_{ab}(x,\widehat z_k)}{\rho_n\,d^t_a(x,\widehat z_k)d^t_b(x,\widehat z_k)}\Bigr)\\
\;&=\;\sum_{1\leq a,b\leq k} \frac{d^t_a(x,\widehat z_k)d^t_b(x,\widehat z_k)}{\rho_n^2n^4} \varphi\left(\frac{\tilde o_{ab}(x,\widehat z_k)/\rho_n n^2}{d^t_a(x,\widehat z_k)d^t_b(x,\widehat z_k)/ \rho_n^2 n^4}\right),
\end{split}
\end{equation}
where $\varphi(u)=u\log u$. Then by Lemma~\ref{lema-sandro}, taking $\epsilon_n = \rho_n = \frac{\log n}{n}$ we have that 
\begin{equation*}
\Bigl| \frac{\tilde o_{ab}(x,\widehat z_k)}{\rho_n n^2}  - [Q_n(\widehat z_k,z)\tilde\lambda Q_n^T(\widehat z_k,z)]_{ab} \Bigr| \; \leq \; \epsilon_n
\end{equation*}
and similarly 
\begin{equation*}
\Bigl| \frac{d^t_a(x,\widehat z_k)}{\rho_n n^2} - [Q_n(\widehat z_k,z)\tilde\lambda Q_n(\widehat z_k,z)^T \mathbf{1}_{k}]_{a} \Bigr| \; \leq \; \epsilon_n
\end{equation*}
eventually almost surely as $n\to\infty$. Then as $\varphi$ is continuous, 
 substituting $\tilde o_{ab}(x,\widehat z_k)/\rho_n n^2$ by $[Q_n(\widehat z_k,z)\tilde\lambda Q_n^T(\widehat z_k,z)]_{ab}$ and $d^t_a(x,\widehat z_k)/\rho_n n^2$ by $[Q_n(\widehat z_k,z)\tilde\lambda Q_n(\widehat z_k,z)^T \mathbf{1}_{k}]_{a}$ in the right-hand side of \eqref{eq2302} we obtain that 
 \begin{equation}\label{equalei40}
\begin{split}
\sum_{1\leq a,  b\leq k}&
\frac{d^t_a(x,\widehat z_k)d^t_b(x,\widehat z_k)}{\rho_n^2n^4}\,\varphi\Bigl(\frac{\rho_n n^2\tilde o_{ab}(x,\widehat z_k)}{d^t_a(x,\widehat z_k)d^t_b(x,\widehat z_k)}\Bigr)\\
&\;\leq\; \;\sup_{\substack {Q_n\colon \|Q_n\|_1=1\\Q_n^T{\bf 1}_k = n(z)/n}}\; 
\sum_{1\leq a, b\leq k} [Q_n\tilde\lambda Q_n^T \mathbf{1}_{k}]_{a}[Q_n\tilde\lambda Q_n^T \mathbf{1}_{k}]_{b} \; \varphi\biggl(\dfrac{[Q_n\tilde\lambda Q_n^T]_{ab}}{ [Q_n\tilde\lambda Q_n^T \mathbf{1}_{k}]_{a}[Q_n\tilde\lambda Q_n^T \mathbf{1}_{k}]_{b}}\biggr) + \eta_n\,.
\end{split}
\end{equation}
for some sequence $\eta_n\to 0$ as $n\to\infty$. Then taking $\lim\sup$ on both sides, we must have that 
\begin{equation}\label{equalei4}
\begin{split}
\limsup_{n\to\infty} \; \sum_{1\leq a,  b\leq k}&
\frac{d^t_a(x,\widehat z_k)d^t_b(x,\widehat z_k)}{\rho_n^2n^4}\,\varphi\Bigl(\frac{\rho_n n^2\tilde o_{ab}(x,\widehat z_k)}{d^t_a(x,\widehat z_k)d^t_b(x,\widehat z_k)}\Bigr)\\
&\;\leq\; \;\sup_{\substack {R\colon \| R\|_1=1\\R^T{\bf 1}_k = \pi}}\; {1\over 2}\sum_{1\leq a, b\leq k} [R\tilde\lambda R^T{\bf 1}_k]_{a}[R\tilde\lambda R^T{\bf 1}_k]_{b} \; \varphi\biggl(\dfrac{[R\tilde\lambda R^{T}]_{ab}}{ [R\tilde\lambda R^T{\bf 1}_k]_{a}[R\tilde\lambda R^T{\bf 1}_k]_{b}}\biggr)
\end{split}
\end{equation}
almost surely. 
Then, the supremum  in the right-hand side of \eqref{equalei4} is a maximum of a convex function over a convex polyhedron defined by $\{R\colon \|R\|_1=1, R^T{\bf 1}_k = \pi \}$. Then, the maximum must be attained at one of the vertices of the polyhedron;  that is, on those matrixes $R$ such that one and only one entry by column is greater than zero, given that $\pi_a>0$ for all $a\in \{1,\dots,k_0\}$. We denote by 
$R^*$ one of these maximums (if there is more than one) and let 
\begin{equation}\label{star_param}
\begin{split}
\pi^*_a&\;=\;[R^*{\bf 1}_{k_0}]_a\,\qquad a\in \{1,\dots,k\}\\
\lambda^*_{ab}&\;=\; \frac{[R^*\tilde\lambda{R^{*^{T}}}]_{ab}}{[R^*{\bf 1}_{k_0}{\bf 1}_{k_0}^T R^{*^T}]_{ab}}\,,\qquad a,b\in \{1,\dots,k\}\,.
\end{split}
\end{equation}
Then 
\begin{equation}\label{equalei5}
\begin{split}
\sup_{\substack {R\colon \| R\|_1=1\\R^T{\bf 1}_k = \pi}}\; {1\over 2}\sum_{1\leq a, b\leq k} [R\tilde\lambda R^T{\bf 1}_k]_{a}&[R\tilde\lambda R^T{\bf 1}_k]_{b} \; \varphi\biggl(\dfrac{[R\tilde\lambda R^{T}]_{ab}}{ [R\tilde\lambda R^T{\bf 1}_k]_{a}[R\tilde\lambda R^T{\bf 1}_k]_{b}}\biggr)\\
&\;=\; {1\over 2}\sum_{1\leq a, b\leq k}  \pi^*_a\pi^*_b[\lambda^*\pi^*]_{a}[\lambda^*\pi^*]_{b} \; \varphi\biggl(\dfrac{\lambda^*_{ab}}{ [\lambda^*\pi^*]_{a}[\lambda^*\pi^*]_{b}}\biggr)\,.
\end{split}
\end{equation}
This concludes the proof of Lemma~\ref{lema-merging}.
\end{proof}

\begin{lemma}\label{lema-merging2}
Assume $\tilde\lambda$ has no two proportional columns. Then for all $k<k_0$ and $(\pi^*,\lambda^*)$ as in Lemma~\ref{lema-merging} we have that 
\begin{equation}\label{equasub234}
\begin{split}
\sum_{1\leq a,  b\leq k_0}\pi_a\pi_b&[\tilde\lambda\pi]_a [\tilde\lambda \pi]_b\;\varphi\Bigl(\frac{\tilde\lambda_{ab}}{[\tilde\lambda\pi]_a [\tilde\lambda \pi]_b}\Bigr)\\
& -\sum_{1\leq a, b\leq k}  \pi^*_a\pi^*_b[\lambda^*\pi^*]_{a}[\lambda^*\pi^*]_{b} \; \varphi\biggl(\dfrac{\lambda^*_{ab}}{ [\lambda^*\pi^*]_{a}[\lambda^*\pi^*]_{b}}\biggr)\;>\; 0\,.
\end{split}
\end{equation}
\end{lemma}
\begin{proof}
First consider the case $k=k_0-1$. As $R^*$ has one and only one non-zero entry in each column, we have that there is a surjective function $h\colon[k_0]\to[k]$ connecting each community in $[k_0]$ (columns of $R^*$) with is corresponding community in $[k]$ (line with non-zero entry). Then for $k=k_0-1$,  there are $k-1$ communities in 
$\{1,\dots,k_0\}$ that are mapped into $k-1$ communities in  $\{1,\dots,k\}$ and two communities in $\{1,\dots,k_0\}$ that are mapped into  a single community in   $\{1,\dots,k\}$. 
Without loss of generality assume that the communities $k_0-1$ and $k_0$ satisfy $h(k_0-1)=h(k_0)=k=k_0-1$. Moreover, as $R^{*T}{\bf 1}_{k}=\pi$ we must have that  the non-zero entries are given by 
\begin{equation}
\begin{split}
R_{aa}^{*}&=\pi_a\,, \qquad 1\leq a\leq k_0-1\\
R_{(k_0-1)k_0}^{*}&= \pi_{k_0}\,.
\end{split}
\end{equation}
Then the parameters $\pi^*$ and $\lambda^*$ defined in \eqref{star_param} are given by 
\begin{equation*}
\begin{split}
\pi_{a}^{*}&=\pi_a\,, \qquad 1\leq a\leq k_0-1 \\
\pi_{k_{0}-1}^{*}&=\pi_{k_0-1}+\pi_{k_0}
\end{split}
\end{equation*}
and 
\begin{align*}
\lambda_{ab}^*&= \tilde\lambda_{ab}\,, \qquad 1\leq a,b\leq k_0-2\\
\lambda_{a(k_0-1)}^*&=\frac{\pi_{k_0-1}\tilde\lambda_{a(k_0-1)}+\pi_{k_0}\tilde\lambda_{ak_0}}{\pi_{k_0-1}+\pi_{k_0}}\,,\qquad1\leq l\leq k_0-2,\nonumber\\
\lambda_{(k_0-1)(k_0-1)}^*&=\frac{\pi_{k_0-1}^2\tilde\lambda_{(k_0-1)(k_0-1)}+2\pi_{k_0-1}\pi_{k_0}\tilde\lambda_{(k_0-1)k_0}+\pi_{k_0}^2\tilde\lambda_{k_0k_0}}{\pi_{k_0-1}^2+2\pi_{k_0-1}\pi_{k_0}+\pi_{k_0}^2}\,.
\end{align*}
Observe that for all  $1\leq a\leq k_0-1$ we have that $[\lambda^*\pi^*]_a=[\tilde\lambda\pi]_a$ then for all $1\leq a,b\leq k_0-2$
\begin{equation*}
\begin{split}
\pi^*_a\pi^*_b[\lambda^*\pi^*]_{a}[\lambda^*\pi^*]_{b} \; \varphi\biggl(\dfrac{\lambda^*_{ab}}{ [\lambda^*\pi^*]_{a}[\lambda^*\pi^*]_{b}}\biggr)\;=\; 
\pi_a\pi_b[\tilde\lambda\pi]_{a}[\tilde\lambda\pi]_{b} \; \varphi\biggl(\dfrac{\tilde\lambda_{ab}}{ [\tilde\lambda\pi]_{a}[\tilde\lambda\pi]_{b}}\biggr)\,.
\end{split}
\end{equation*}
On the other hand we have that 
\[
[\lambda^*\pi^*]_{k_0-1} = \frac{\pi_{k_0-1}[\tilde\lambda\pi]_{k_0-1}+\pi_{k_0}[\tilde\lambda\pi]_{k_0} }{\pi_{k_0-1}+\pi_{k_0}}\,.
\]
Then for $1\leq a\leq k_0-2$ it follows, by the log-sum inequality,   that 
\begin{equation}\label{equasub180}
\begin{split}
\pi_a^*&\pi_{k_0-1}^*[\lambda^*\pi^*]_{a}[\lambda^*\pi^*]_{k_0-1}\varphi\biggl(\dfrac{\lambda^*_{a(k_0-1)}}{ [\lambda^*\pi^*]_{a}[\lambda^*\pi^*]_{k_0-1}}\biggr)\\
&=\;\pi_a[\tilde\lambda\pi]_a(\pi_{k_0-1}[\tilde\lambda\pi]_{k_0-1}+\pi_{k_0}[\tilde\lambda\pi]_{k_0})\varphi\left(\frac{\pi_{k_0-1}\tilde\lambda_{a(k_0-1)}+\pi_{k_0}\tilde\lambda_{ak_0}}{[\tilde\lambda\pi]_a(\pi_{k_0-1}[\tilde\lambda\pi]_{k_0-1}+\pi_{k_0}[\tilde\lambda\pi]_{k_0})}\right)\\
&=\;\pi_a(\pi_{k_0-1}\tilde\lambda_{a(k_0-1)}+\pi_{k_0}\tilde\lambda_{ak_0})\log\left(\frac{\pi_a(\pi_{k_0-1}\tilde\lambda_{a(k_0-1)}+\pi_{k_0}\tilde\lambda_{ak_0})}{\pi_a[\tilde\lambda\pi]_a(\pi_{k_0-1}[\tilde\lambda\pi]_{k_0-1}+\pi_{k_0}[\tilde\lambda\pi]_{k_0})}\right)\\
&\leq\; \pi_a\pi_{k_0-1}\tilde\lambda_{a(k_0-1)} \log\biggl(\frac{\tilde\lambda_{a(k_0-1)}}{[\tilde\lambda\pi]_a[\tilde\lambda\pi]_{k_0-1}}\biggr)+\pi_a\pi_{k_0} \tilde\lambda_{ak_0}   \log\biggl(\frac{\tilde\lambda_{ak_0}}{[\tilde\lambda\pi]_a[\tilde\lambda\pi]_{k_0}}\biggr)\\
&=\; \pi_a\pi_{k_0-1}[\tilde\lambda\pi]_a[\tilde\lambda\pi]_{k_0-1}\varphi\biggl(\frac{\tilde\lambda_{a(k_0-1)}}{[\tilde\lambda\pi]_a[\tilde\lambda\pi]_{k_0-1}}\biggr)+\pi_a\pi_{k_0}[\tilde\lambda\pi]_a[\tilde\lambda\pi]_{k_0}\varphi\biggl(\frac{\tilde\lambda_{ak_0}}{[\tilde\lambda\pi]_a[\tilde\lambda\pi]_{k_0}}\biggr)\,.
\end{split}
\end{equation}
Moreover, we have that the inequality must be strict unless
\begin{equation}\label{lambdapi}
\frac{\tilde\lambda_{a(k_0-1)}}{[\tilde\lambda\pi]_a[\tilde\lambda\pi]_{k_0-1}} \;=\; \frac{\tilde\lambda_{ak_0}}{[\tilde\lambda\pi]_a[\tilde\lambda\pi]_{k_0}} \qquad \text{for all }a\leq k_0-2\,.
\end{equation}
On the other hand, for $a= k_0-1$ and $b=k_0-1$, also by the log-sum inequality  we have that 
\begin{equation}\label{equasub19}
\begin{split}
&\pi_{k_0-1}^*\pi_{k_0-1}^*[\lambda^*\pi^*]_{k_0-1}[\lambda^*\pi^*]_{k_0-1}\varphi\biggl(\dfrac{\lambda^*_{(k_0-1)(k_0-1)}}{ [\lambda^*\pi^*]_{k_0-1}[\lambda^*\pi^*]_{k_0-1}}\biggr)\\
&=(\pi_{k_0-1}^2\tilde\lambda_{(k_0-1)(k_0-1)}+2\pi_{k_0-1}\pi_{k_0}\tilde\lambda_{(k_0-1)k_0}+\pi_{k_0}^2\tilde\lambda_{k_0k_0} )\\
&\qquad\qquad \times \log\biggl(\dfrac{\pi_{k_0-1}^2\tilde\lambda_{(k_0-1)(k_0-1)}+2\pi_{k_0-1}\pi_{k_0}\tilde\lambda_{(k_0-1)k_0}+\pi_{k_0}^2\tilde\lambda_{k_0k_0}}{\pi_{k_0-1}^2[\tilde\lambda\pi]_{k_0-1}^2+2\pi_{k_0-1}\pi_{k_0}[\tilde\lambda\pi]_{k_0-1}[\tilde\lambda\pi]_{k_0}+\pi_{k_0}^2[\tilde\lambda\pi]_{k_0}^2}\biggr)\\
&\leq\; 
 \pi_{k_0-1}^2\tilde\lambda_{(k_0-1)(k_0-1)}\log\Bigl( \frac{\tilde\lambda_{(k_0-1)(k_0-1)}}{[\tilde\lambda\pi]_{k_0-1}^2}\Bigr) + 
  2\pi_{k_0-1}\pi_{k_0}\tilde\lambda_{(k_0-1)k_0}\log\Bigl( \frac{\tilde\lambda_{(k_0-1)k_0}}{[\tilde\lambda\pi]_{k_0-1}[\tilde\lambda\pi]_{k_0}}\Bigr) \\
  & \qquad + \pi_{k_0}^2\tilde\lambda_{k_0k_0}\log\Bigl( \frac{\tilde\lambda_{k_0k_0}}{[\tilde\lambda\pi]_{k_0}^2}\Bigr)\\
&\; =\pi_{k_0-1}^2[\tilde\lambda\pi]_{k_0-1}^2\varphi\biggl(\frac{\tilde\lambda_{(k_0-1)(k_0-1)}}{[\tilde\lambda\pi]_{k_0-1}^2}\biggr)+2\pi_{k_0-1}\pi_{k_0}[\tilde\lambda\pi]_{k_0-1}[\tilde\lambda\pi]_{k_0}\varphi\biggl(\frac{\tilde\lambda_{(k_0-1)k_0}}{[\tilde\lambda\pi]_{k_0-1}[\tilde\lambda\pi]_{k_0}}\biggr)\\
&\qquad +\pi_{k_0}^2[\tilde\lambda\pi]_{k_0}^2\varphi\biggl(\frac{\tilde\lambda_{k_0k_0}}{[\tilde\lambda\pi]_{k_0}^2}\biggr)\,,
\end{split}
\end{equation}
with equality if and only if
\begin{equation}\label{lambdapi2}
\frac{\tilde\lambda_{(k_0-1)(k_0-1)}}{[\tilde\lambda\pi]_{k_0-1}^2}\;=\; \frac{\tilde\lambda_{(k_0-1)k_0}}{[\tilde\lambda\pi]_{k_0-1}[\tilde\lambda\pi]_{k_0}}\;=\;\frac{\tilde\lambda_{k_0k_0}}{[\tilde\lambda\pi]_{k_0}^2}\,.
\end{equation}
From \eqref{lambdapi} and \eqref{lambdapi2} we obtain that the inequality \eqref{equasub234} must be strict unless 
\begin{equation}\label{lambdapi3}
\tilde\lambda_{a(k_0-1)} \;=\;  \frac{[\tilde\lambda\pi]_{k_0-1}}{[\tilde\lambda\pi]_{k_0}} \tilde\lambda_{ak_0} \qquad \text{for all }a\leq k_0\,
\end{equation}
which is a contradiction with the hypothesis for the  identifiability of $k_0$. 
\end{proof}

\end{document}